\documentclass[a4paper, 12pt, oneside, notitlepage]{amsart}
\usepackage[margin=3cm]{geometry}
\usepackage{amsmath,amssymb,amsthm,graphicx,mathrsfs,bbm,url}
\usepackage{amsthm}
\usepackage{wrapfig}
\usepackage{enumitem}
\usepackage{mathtools}
\usepackage[utf8]{inputenc}
 \usepackage[T1]{fontenc}
\usepackage[usenames,dvipsnames]{color}
\usepackage[colorlinks=true,linkcolor=Red,citecolor=Green]{hyperref}
\usepackage[super]{nth}
\usepackage[open, openlevel=2, depth=3, atend]{bookmark}
\hypersetup{pdfstartview=XYZ}
\usepackage[font=footnotesize]{caption}
\usepackage{a4wide}
\usepackage{tikz-cd}

\usepackage{epstopdf}
\usepackage{hyperref}

\theoremstyle{plain}
\newtheorem{theorem}{Theorem}[section]
\newtheorem*{theorem*}{Theorem}

\newtheorem{lemma}[theorem]{Lemma}
\newtheorem{proposition}[theorem]{Proposition}
\newtheorem{corollary}[theorem]{Corollary}

\newtheorem{theoremABC}{Theorem}

\newtheorem*{question}{Question}
\theoremstyle{definition}
\newtheorem{definition}[theorem]{Definition}

\theoremstyle{remark}
\newtheorem{remark}[theorem]{Remark}

\numberwithin{equation}{section}

\title
[Stability and uniqueness of minimal disks in non-constant curvature]
{Stability and uniqueness of minimal disks in non-constant curvature}

\author[Alvarez]{S\'ebastien Alvarez}
\address{CMAT, Facultad de Ciencias, Universidad de la Rep\'ublica, \& IRL-IFUMI (CNRS)\\
	Igu\'a 4225 esq. Mataojo. Montevideo, Uruguay.\\}
\email{salvarez@cmat.edu.uy}

\author[Lefeuvre]{Thibault Lefeuvre}
\address{Université Paris-Saclay, CNRS, Laboratoire de mathématiques d’Orsay, 91405, Orsay, France.}
\email{thibault.lefeuvre1@universite-paris-saclay.fr}

\author[Lowe]{Ben Lowe}
\address{University of Chicago, Department of Mathematics, Chicago IL 60637, USA}
\email{loweb24@gmail.com}

\author[Smith]{Graham A. Smith}
\address{Departamento de Matemática, PUC-Rio, Rua Marquês de São Vicente, 225, Rio de Janeiro, 22451-900, RJ, Brazil.}
\email{grahamandrewsmith@gmail.com}

\catcode`\@=11
\def\eqalign#1{\null\,\vcenter{\openup1\jot \m@th %
\ialign{\strut\hfil$\displaystyle{##}\quad$&$\displaystyle{{}##}$\hfil\crcr#1\crcr}}\,}
\def\triplealign#1{\null\,\vcenter{\openup1\jot \m@th %
\ialign{\strut\hfil$\displaystyle{##}\quad$&$\displaystyle{{}##}$\hfil&$\displaystyle{{}##}$\hfil\crcr#1\crcr}}\,}
\def\multiline#1{\null\,\vcenter{\openup1\jot \m@th %
\ialign{\strut$\displaystyle{##}$\hfil&$\displaystyle{{}##}$\hfil\crcr#1\crcr}}\,}
\catcode`\@=12

\newcommand{\closedthreeball}{\overline{\mathbb{B}}\vphantom{\mathbb{B}}^3}
\newcommand{\C}{\mathbb{C}}

\newcommand{\Spec}{\text{{\rm Spec}}}
\newcommand{\Id}{\text{{\rm Id}}}
\newcommand{\sect}{\text{{\rm sect}}}

\newcommand{\tr}{\text{{\rm tr}}}

\newcommand{\opD}{{\text{{\rm D}}}}

\newcommand{\opSupp}{\text{{\rm Supp}}}
\newcommand{\opLength}{\text{{\rm Length}}}

\newcommand{\opR}{{\text{{\rm R}}}}
\newcommand{\opT}{{\text{{\rm T}}}}
\newcommand{\opTr}{\text{{\rm Tr}}}

\newcommand{\overlineint}{{\text{{\rm I}}}\overline{\text{{\rm nt}}}}
\newcommand{\overlineR}{\overline{\text{{\rm R}}}}
\newcommand{\overlineRic}{\text{{\rm R}}\overline{\text{{\rm\i c}}}}

\newcommand{\opU}{{\text{{\rm U}}}}

\newcommand{\opArea}{{\text{{\rm Area}}}}

\newcommand{\sphere}{{\mathbb{S}}}

\newcommand{\opMI}{{\mathrm{MI}}}

\newcommand{\opDeg}{{\mathrm{Deg}}}

\newcommand{\dd}{\mathrm{d}}

\newif\ifshowcomments
\showcommentstrue
\def\gs#1{\ifshowcomments{\textcolor{red}{GS: #1}}\fi}
\def\sa#1{\ifshowcomments{\textcolor{blue}{SA: #1}}\fi}

\def\tl#1{\ifshowcomments{\textcolor{teal}{TL: #1}}\fi}
\def\m#1{\ifshowcomments{\textcolor{Mahogany}{ #1}}\fi}
\begin{document}


\begin{abstract}
Nitsche proved that every smooth Jordan curve in $\mathbb{R}^3$ of total
curvature at most $4\pi$ bounds a unique minimal disk, which is moreover
strictly stable. We prove an analogue of this result for Riemannian
$3$-balls with mean convex boundary, under an explicit pinching condition on
the negative sectional curvature, together with a bound on the covariant
derivative of the Ricci tensor. In this setting, every smooth Jordan curve
in the boundary sphere of total curvature at most $4\pi$ bounds a unique embedded
minimal disk which is strictly stable.
\end{abstract}

\maketitle

\section{Introduction}

\subsection{Context}
In this paper we investigate uniqueness and stability of immersed minimal disks with prescribed boundary in Riemannian $3$-manifolds. A fundamental result in this direction was established by Nitsche \cite{Nitsche1973}, who proved that \emph{every smooth Jordan curve in $\mathbb{R}^3$ with total curvature at most $4\pi$ bounds a unique immersed minimal disk, which is moreover strictly stable.} This disk was later shown to be embedded by \cite{EkholmWhiteWienholtz2002}.

Here \emph{strict stability} means that the minimal disk $D$ is a strict local minimizer of the area functional among normal variations fixing $\partial D$. Equivalently, this amounts to the positivity of the first \emph{Dirichlet eigenvalue} of its \emph{Jacobi operator}, which measures the second variation of area with respect to normal variations and which we will introduce presently.

Barbosa–-do Carmo \cite{BarbosaDoCarmo1980} extended the stability part of Nitsche's theorem to space forms of constant sectional curvature, and uniqueness was later obtained in that setting by Li--Jost in \cite{Li-Jost}. Note that, without additional geometric conditions on the boundary curve, the problem of uniqueness is specific to the negatively curved setting. Indeed, in $\mathbb{S}^3$, for example, the equatorial sphere has vanishing total curvature, yet bounds a continuum of minimal disks.

Barbosa–do Carmo also indicated that their method should extend to general Riemannian ambient manifolds, but noted that the explicit curvature estimate it required was ``far from being completely settled'' outside the constant curvature case \cite[\S1.7]{BarbosaDoCarmo1980}. We establish this estimate assuming pinched negative curvature, with an explicit pinching threshold. Furthermore, under these hypotheses, we also obtain uniqueness within the class of embedded minimal disks.

\subsection{Main results} Before stating our results, we introduce some notation. Let $\overline g$ be a Riemannian metric on the closed unit ball $\closedthreeball$, let $\overline\nabla$ denote its \emph{Levi-Civita connection}, let $\overline{\sect}$ denote its
\emph{sectional curvature}, and let $\overlineRic$ denote its \emph{Ricci tensor}. We denote by $\|\overline\nabla\overlineRic\|$ the \emph{Hilbert--Schmidt norm} of $\overline{\nabla}\overlineRic$ (see \S\ref{ss.curvatures}). Finally, we define the \emph{total curvature} of a smooth closed curve $c$ in $\closedthreeball$ by
\begin{equation*}
K_c:=\int_c|\boldsymbol\kappa_c|\,\dd c\ ,
\end{equation*}
where $\boldsymbol\kappa_c=\overline\nabla_TT$ denotes its curvature vector, and $\dd c$ denotes its arc-length measure.

\begin{theoremABC}
\label{thm:MainResultA}
Let $$\alpha_\star:=\frac{4\sqrt{17}-8}{13}=0.653263269\ldots.$$
Suppose that $(\closedthreeball,\overline{g})$ has mean convex boundary and that, for some $\kappa>0$,
\begin{equation}
\label{eqn:CurvatureConditions}
-(1+\alpha_\star)\kappa^2\leq\overline{\sect}\leq-\kappa^2\qquad\text{and}\qquad\|\overline{\nabla}\overlineRic\| \leq \alpha_\star\kappa^3\ .
\end{equation}
If $c$ is a smooth, simple, closed curve in $\sphere^2$ such that
\begin{equation*}
K_c \leq 4\pi\ ,
\end{equation*}
then $c$ bounds a unique, embedded minimal disk $D$ in $(\closedthreeball,\overline{g})$. Furthermore, this disk is strictly stable.
\end{theoremABC}

\begin{remark}
Under the mean convex boundary assumption, existence of a stable, embedded minimal disk is a consequence of Meeks-Yau theorem \cite{MeeksYau1982}. The novelty of Theorem \ref{thm:MainResultA} is thus uniqueness and \emph{strict} stability.
\end{remark}
\begin{remark}
Our hypotheses are not sharp. The quantity $\alpha_\star$ is merely the threshold at which our techniques break down. This is also reflected in the non-strict inequalities used. Indeed, \emph{strict} stability, and hence also uniqueness, are guaranteed since the operator $J$ has spectrum \emph{strictly} greater than that of $J'$ in the proof of Theorem \ref{thm:Nitsche} in Section \ref{ss:ProofOfNitsche}, below.
\end{remark}

Theorem \ref{thm:MainResultA} is a consequence of the following generalization of Nitsche's theorem.
\begin{theoremABC}
\label{thm:Nitsche}
Under the same hypotheses as in Theorem \ref{thm:MainResultA}, every immersed minimal disk $D$ in $(\closedthreeball,\overline{g})$ bounded by $c$ is strictly stable.
\end{theoremABC}

\begin{remark}
The bound $4\pi$ in Nitsche's theorem is believed to be \emph{sharp}. Indeed, in \cite{Bohme82}, B\"ohme announced that, for all $\epsilon>0$ and $n>1$, there exists a simple closed curve of total curvature at most $4\pi+\epsilon$ bounding at least $n$ minimal disks. However, the proof was never published, and it is possible that the disks in question have branch points. In \cite{Nitsche68}, Nitsche produced \emph{explicit} counterexamples to uniqueness . Indeed, recall that the \emph{Enneper surface} is the image of the map $z\mapsto(\text{Re}(z-\frac{z^3}{3}),-\text{Im}(z+\frac{z^3}{3}),\text{Re}(z^2))$. For $r\in(1,\sqrt{3})$, the image under this map of $\{|z|\leq r\}$ is an \emph{embedded} minimal disk, and since its Jacobi operator is non-degenerate of Morse index $1$, it is unstable. Its boundary, called the \emph{Enneper wire} is a Jordan curve $c_r$ which lies on an explicit convex ellipsoid $E_r$ (see \cite[\S 2.8]{Nitsche68}). In particular, since it lies on the boundary of a convex body, by the result \cite{MeeksYau1982} of Meeks--Yau, $c_r$ also bounds an embedded, area-minimizing, and thus \emph{stable}, disk. In this manner, Nitsche showed that each Enneper wire spans at least two distinct \emph{embedded minimal disks}. Nitsche also computed the total curvatures of Enneper wires explicitly in terms of elliptic integrals \cite[p.~437]{Nitsche_book}. For $r$ close to $1$, this total curvature is roughly $4.24\pi$.

The property of bounding a non-degenerate embedded minimal disk of Morse index $1$ is stable under perturbations of the ambient metric. This can be proved with an argument similar to the proof of Lemma \ref{lemma:ImplicitFunctionTheorem} in Section 5 below. Likewise, the convexity of the ellipsoid containing the Enneper wire is also stable. By perturbing the Euclidean metric to a metric of small constant negative curvature, and applying Meeks--Yau once again, we therefore obtain a simple closed curve, on the boundary of a convex domain in a negatively curved manifold, of total curvature at most $4.25\pi$, which bounds two distinct embedded minimal disks, one of Morse index $1$ and one area-minimizing.
\end{remark}

\subsection{Foliated Plateau problems and boundary area spectrum} 
Theorem~\ref{thm:MainResultA} is a key ingredient in our companion paper \cite{AlvarezLefeuvreLoweSmith}, where we introduce \emph{area simple metrics}, define their boundary area spectrum, and study its local and global rigidity.

In that reference, we consider minimal discs bounded by \emph{round circles}, where a round circle is the intersection of $\mathbb{S}^2$ with an affine plane in $\mathbb{R}^3$. We denote by $\mathcal{C}$ the set of oriented round circles, and we say that a Riemannian metric $\bar g$ on $\closedthreeball$ is \emph{area simple} if
\begin{enumerate}
  \item $\mathbb{S}^2$ is mean convex with respect to $\bar g$;
  \item every $c\in\mathcal{C}$ bounds a unique $\bar g$-minimal disc $\opD_{\bar g}(c)$, and this disc is embedded and strictly stable.
\end{enumerate}
The model example is a geodesic ball in hyperbolic space, identified with $\closedthreeball$ via the Klein model (suitably rescaled): $\opD_{\bar g}(c)$ is then simply the flat disc bounded by $c$, which is totally geodesic. Area simple metrics form an open subset of the space of Riemannian metrics on $\closedthreeball$ \cite{AlvarezLefeuvreLoweSmith}. The \emph{boundary area spectrum} of an area simple metric $\bar g$ is the map
\[
  \mathcal{A}_{\bar g}\colon\mathcal{C}\to(0,\infty),\qquad c\mapsto\opArea_{\bar g}\bigl(\opD_{\bar g}(c)\bigr).
\]
The \emph{boundary area rigidity problem} asks whether $\mathcal{A}_{\bar g}$ determines $\bar g$, up to isometry. In \cite{AlvarezLefeuvreLoweSmith}, we prove several rigidity results, both local and global, for pairs of conformal metrics $\bar g$ and $e^{2f}\bar g$. In particular, if both are area simple and $\mathcal{A}_{\bar g}=\mathcal{A}_{e^{2f}\bar g}$, then $f=0$ in each of the following two cases:
\begin{enumerate}
  \item \textbf{Local rigidity:} \emph{$\bar g$ belongs to a certain open and dense subset of the space of area simple metrics, endowed with the $C^\infty$ topology, and $f\in C^{14}(\closedthreeball)$ satisfies $\|f\|_{C^{14}}\leq\epsilon$, where $\epsilon>0$ depends only on $\bar g$};
  \item \textbf{Global rigidity:} \emph{$\bar g$ and $e^{2f}\bar g$ are real analytic.}
\end{enumerate}

Theorem~\ref{thm:MainResultA} provides an explicit open set of area simple metrics. More precisely, we prove that every Riemannian metric satisfying the following conditions for some $\kappa>0$ is area simple
\begin{enumerate}
  \item[(A1)] \emph{$\mathbb{S}^2$ is mean convex with respect to $\bar g$};
  \item[(A2)] \emph{$\text{K}(\bar g)<4\pi$, where $\text{K}(\bar g)$ denotes the supremum of the total curvatures of the round circles $c\in\mathcal{C}$ with respect to $\bar g$};
  \item[(A3)] $-(1+\alpha_\star)\kappa^2\leq\overline{\sect}\leq-\kappa^2$;
  \item[(A4)] $\|\overline{\nabla}\overlineRic\| \leq \alpha_\star\kappa^3$.
\end{enumerate}

\subsection{Asymptotic Plateau problems} It is natural to ask whether Theorem \ref{thm:MainResultA} admits an analogue for the asymptotic Plateau problem. Given a Jordan curve $\Lambda$ in the boundary at infinity of a simply connected negatively curved $3$-manifold, one asks whether it bounds a unique properly embedded minimal disk.

Existence was established by Anderson \cite{Anderson83} (see also \cite{Som04,soma05}), who also showed the failure of uniqueness in general by constructing Jordan curves bounding infinitely many complete embedded minimal surfaces. However, uniqueness is known under additional geometric assumptions. In \cite{Uhlenbeck83}, Uhlenbeck proved uniqueness in the almost-Fuchsian case, whilst in \cite{Seppi16}, Seppi established explicit conditions on the quasisymmetry constant of the boundary curve for Uhlenbeck's hypotheses to hold. More recently, Huang--Lowe--Seppi \cite{HuangLoweSeppi26} established uniqueness under quantitative assumptions on the geometry at infinity, such as finite width or small quasiconformal distortion, together with suitable control of the asymptotic geometry of the minimal disk. On the other hand, they also constructed Jordan curves bounding uncountably many distinct properly embedded stable minimal disks.

Unlike the classical Plateau problem, however, the boundary at infinity carries only a conformal structure, so that total curvature is no longer defined. The known uniqueness criteria are therefore formulated in terms of different geometric quantities, such as the width of the curve or its quasisymmetry constant. It would be interesting to identify a geometric condition at infinity playing the same role as the total curvature bound in Nitsche's theorem and in Theorem \ref{thm:MainResultA}.

\begin{question}
Does there exist a natural geometric condition on Jordan curves at infinity yielding uniqueness for the asymptotic Plateau problem in $\mathbb H^3$, in asymptotically hyperbolic $3$-manifolds \cite{alexakis2010renormalized,marx2024inverse}, or in universal covers of closed negatively curved $3$-manifolds \cite{Gromov-91-1,lowe2021deformations}?
\end{question}

\noindent \textbf{GenAI disclosure:} The main objective of Theorem \ref{thm:MainResultA}, namely to obtain stability over some explicit neighbourhood of the hyperbolic metric, was established without the assistance of GenAI. Our initial approach yielded a weaker pinching constant $\alpha_\star$. With the help of GenAI, we obtained Proposition \ref{prop_bound_ricci}, which yields an improved value of this constant. Theorem \ref{thm:Nitsche} was derived from Theorem \ref{thm:MainResultA} without the assistance of GenAI. GenAI was used to assist in language and presentation.\\
\\
\noindent \textbf{Acknowledgement:} This project has received funding from the European Research Council (ERC) under the European Union’s Horizon research and innovation programme (grant agreement No. 101162990) and from IRL-2030 IFUMI, Laboratorio del Plata. S.A. acknowledges financial support from CSIC through the
Grupo I+D 149-348 ``Geometría y Acciones de Grupos''. During the preparation of this
paper, S.A. benefited from a \emph{poste rouge} from the CNRS.

\section{Geometric preliminaries}

\subsection{Curvature}\label{ss.curvatures} Let $(\closedthreeball,\overline{g})$ be a closed Riemannian $3$-ball. Let $\overline\nabla$ denote the Levi-Civita connection. We use the following convention for the Riemann curvature tensor
\begin{equation}\label{eq:curv-convention}
 \overlineR(X,Y)Z
 :=
 \overline\nabla_X\overline\nabla_YZ
 -\overline\nabla_Y\overline\nabla_XZ
 -\overline\nabla_{[X,Y]}Z\ ,
\end{equation}
and we denote
\begin{equation*}
\overlineR(X,Y,Z,W):=\overline g(\overlineR(X,Y)Z,W)\ .
\end{equation*}
We recall that this $(0,4)$-tensor satisfies the following symmetries
\begin{equation*}
\overlineR(X,Y,Z,W)=-\overlineR(Y,X,Z,W)=-\overlineR(X,Y,W,Z)=\overlineR(Z,W,X,Y)\ .
\end{equation*}

Given an orthonormal pair of vectors $(X,Y)$, we define the sectional curvature of the plane that they generate by
\begin{equation}\label{eq:sect-def}
\overline{\text{sect}}(X,Y):=\overlineR(X,Y,Y,X)\ .
\end{equation}

We define the \emph{Ricci tensor} to be the trace
\begin{equation}\label{eq:Ric-def}
\overlineRic(Y,Z):=\tr_{\overline g}\bigl(X\longmapsto\overlineR(X,Y)Z\bigr)=\sum_k\overlineR(f_k,Y,Z,f_k)\ ,
\end{equation}
where $(f_k)$ is any orthonormal frame of $T\closedthreeball$. Its covariant derivative is the $(0,3)$-tensor
\begin{equation}\label{eq:nablaRic}
(\overline\nabla\overlineRic)(X;Y,Z):=(\overline\nabla_X\overlineRic)(Y,Z)\ .
\end{equation}
In what follows, we use the \emph{Hilbert-Schmidt norm} of $\overline{\nabla}\overlineRic$, given by
\begin{equation}\label{eq:nablaRic-norm}
|\overline\nabla\overlineRic|_p^2
:=\sum_{a,b,c}\bigl((\overline\nabla_{f_a}\overlineRic)(f_b,f_c)\bigr)^2\ ,
\end{equation}
which is independent of the orthonormal frame chosen. The norm over $\overline{\nabla}\overlineRic$ over $\closedthreeball$ is then given as
\begin{equation*}
\|\overline\nabla\overlineRic\|:=\sup_{p\in\closedthreeball}|\overline\nabla\overlineRic|_p\ .
\end{equation*}

\subsection{Geometry of immersed disks} Let $D\subset\closedthreeball$ be an oriented immersed disk. Let $\nu$ denote the unit normal vector field compatible with the chosen orientation. Let $g$ denote the restriction of $\overline g$ to $TD$, let $\nabla$ denote its \emph{Levi--Civita connection}, and let $\opR$ denote its \emph{Riemann curvature tensor}.

\subsubsection{Induced curvatures}
Let $\overline \sigma$ denote the restriction to $TD$ of the \emph{ambient sectional curvature}, that is
\begin{equation}\label{eq.sigmabar}
\overline{\sigma}:=\overline{\text{sect}}(f_1,f_2)\ ,
\end{equation}
where $(f_k)$ is any orthonormal frame of $TD$. Let $\sigma$ denote the \emph{Gaussian curvature} of the \emph{induced metric}, that is
\begin{equation*}
\sigma=\opR(f_1,f_2,f_2,f_1)\ .
\end{equation*}
Recall that, in $2$ dimensions, the Riemann curvature tensor is given by
\begin{equation}\label{eq.Riemann_tensor_dim2}
\opR(X,Y)Z=\sigma(g(Y,Z)X-g(X,Z)Y)\ .
\end{equation}

\subsubsection{Shape operator} The \emph{shape operator} of the disk $D$ is given by
\begin{equation*}
A X:=\overline{\nabla}_X\nu\ .
\end{equation*}
This is a bundle end\m{o}morphism of $TD$ which is self-adjoint with respect to the induced metric $g$. By abuse of notation we will also denote by $A$ the associated \emph{second fundamental form},
\begin{equation*}
A(X,Y)=\bar g(AX,Y)\ .
\end{equation*}

We denote the eigenvalues of $A$, also known as the \emph{principal curvatures} of $D$, by $\lambda_1,\lambda_2$. We denote their corresponding unit eigenvectors, also known as \emph{principal directions}, by $e_1,e_2$, so that, for each $i$,
\begin{equation}\label{eq_pcpl_dir}
A e_i:=\overline{\nabla}_{e_i}\nu=\lambda_ie_i\ .
\end{equation}

We define the \emph{mean curvature} of the disk $D$ by
\begin{equation}\label{eq_mean_curv}
H:=\frac{\tr_g(A)}{2}=\frac{\lambda_1+\lambda_2}{2}\ ,
\end{equation}
and we define its \emph{extrinsic curvature} by
\begin{equation}\label{eq_ext_curv}\kappa_{ext}:=\det(A)=\lambda_1\lambda_2\ .
\end{equation}
Recall that the connections $\overline{\nabla}$ and $\nabla$ are related through the \emph{Gauss' formula} (see \cite{spivak1979comprehensive3}). With our sign conventions, this is
\begin{equation}\label{eq.GW_formula}
\overline{\nabla}_XY=\nabla_XY-A(X,Y)\nu\quad\text{for every }X,Y\in TD\ .
\end{equation}
Likewise, \emph{Codazzi--Mainardi equation} is
\begin{equation}\label{Codazzi_Mainardi1}(\nabla_X A)Y-(\nabla_Y A)X=\overlineR(X,Y)\nu\ .
\end{equation}
Note that the right-hand is tangent to $D$, since its normal component is equal to $\overlineR(X,Y,\nu,\nu)=0$. With our conventions, \emph{Gauss equation} is
\begin{equation}\label{GaussEquation}
\overlineR(X,Y,Z,W)=\opR(X,Y,Z,W)+A(X,Z)A(Y,W)-A(Y,Z)A(X,W)\ .
\end{equation}
In $2$-dimensions, taking $X=W=e_1$ and $Y=Z=e_2$, where $e_1,e_2$ denote the principal principal directions, this reduces to
\begin{equation}\label{GaussEqn2}
\sigma=\kappa_{ext}+\overline{\sigma}\ .
\end{equation}

We will also make use of the following \emph{Ricci identity} for $(0,2)$-tensors
\begin{equation}\label{eq.ricci_identity}
(\nabla_X\nabla_Y-\nabla_Y\nabla_X-\nabla_{[X,Y]})A(Z,W)=
-A(\opR(X,Y)Z,W)-A(Z,\opR(X,Y)W)\ .
\end{equation}

\subsubsection{Adapted frames and pointwise identities} All pointwise estimates below will be given with respect to an \emph{adapted frame} of $D$, which is defined as follows. Given $p\in D$, let $(e_1,e_2)$ be an orthonormal basis of principal directions of $D$ at $p$, that is, for each $i$,
\begin{equation*}
Ae_i = \lambda_i e_i\ .
\end{equation*}
We extend this to a frame over a neighbourhood of $p$ by parallel transport along geodesics leaving $p$, and we call the triplet $(e_1,e_2,\nu)$ an \emph{adapted frame} of $\opT\closedthreeball$ along $D$.

By construction, for all $i$, $j$,
\begin{equation*}
(\nabla_{e_i} e_j)(p) = 0\ ,
\end{equation*}
Thus, by Gauss' formula \eqref{eq.GW_formula}, for all $i$, $j$,
\begin{equation*}
(\overline\nabla_{e_i} e_j)(p) = -\lambda_i\delta_{ij}\nu(p)\ .
\end{equation*}
Furthermore, since $\nabla$ is torsion free,
\begin{equation*}
[e_1,e_2](p) = (\nabla_{e_i}e_j)(p) - (\nabla_{e_2}e_1)(p) = 0\ .
\end{equation*}

Given any $(0,m)$-tensor $S$, we denote
\begin{equation*}
S_{i_1\ldots i_m}:=S(e_{i_1},\cdots,e_{i_m})\ ,
\end{equation*}
and we denote
\begin{equation}
\label{eqn:CovariantDerivativeTerminology}
\nabla_{k_1}\cdots\nabla_{k_n} S_{i_1\cdots i_m} := (\nabla^m S)(e_{i_1},\cdots,e_{i_m};e_{k_n},\cdots,e_{k_1})\ .
\end{equation}
We will require the following pointwise identities concerning the ambient Ricci and Riemann curvature tensors.
\begin{lemma}\label{lem:star}
In an adapted frame, for $\{i,k\}=\{1,2\}$,
\begin{enumerate}
\item $\overlineRic_{ii}=\overline{\sigma}+\overline{\sect}_{i\nu}$;
\item $\overlineRic_{\nu \nu}=\overline{\sect}_{1\nu}+\overline{\sect}_{2\nu}$; and
\item $\overlineR_{ki\nu i}=-\overlineRic_{\nu k}\cdot(1-\delta_{ik})$
\end{enumerate}
\end{lemma}

\subsection{Minimal disks, Jacobi operators and stability}

\subsubsection{Jacobi operator of minimal disks} We say that an immersion $e:D\to\closedthreeball$ is \emph{minimal} whenever its mean curvature vanishes, that is
\begin{equation*}
H=\frac{\lambda_1+\lambda_2}{2}=0\ .
\end{equation*}

The \emph{Jacobi operator} of $e$ is defined as follows. We first extend $(\closedthreeball,\overline g)$ to a slightly larger open manifold $(M,\overline g)$. For sufficiently small $\varepsilon>0$, consider the map
\begin{equation*}
E:D\times(-\varepsilon,\varepsilon)\to M\ ,\qquad (p,t)\mapsto \exp_{e(p)}(t\,\nu(p))\ .
\end{equation*}
For sufficiently small $\varepsilon$, this map is an immersion. More precisely, it is the \emph{Fermi paramet\-rization} of a neighbourhood of $e(D)$ in $M$. Given $u\in C^2(D)$ with $\|u\|_{L^\infty}<\varepsilon$, we define the immersion $e_u:D\to M$ by
\begin{equation*}
e_u(p):=E(p,u(p))\ .
\end{equation*}
In this manner, immersions sufficiently close to $e$ are represented uniquely by normal graphs of functions over $D$. For each $u$, we denote by $A_{u}$ the shape operator of the immersion $e_u$, and we denote by $H_u$ its mean curvature, that is
\begin{equation*}
H_u := \frac{1}{2}\opTr(A_u)\ .
\end{equation*}
The map $(u\mapsto H_u)$ is a second-order, quasi-linear partial differential operator defined over a neighbourhood of $0$ in $C^2(D)$.

\begin{definition}[Jacobi operator]
The \emph{Jacobi operator} $J$ of $D$ is defined to be linearization at $e$ of the mean curvature operator, that is, for all $u\in C^2(D)$,
\begin{equation*}
Ju := \partial_t (2H_{tu})|_{t=0}\ .
\end{equation*}
\end{definition}

The Jacobi operator is a second-order, linear, elliptic partial differential operator over $D$. With our sign conventions, it is given by (see, for example \cite{ColdingMinicozzi2011,Rosenberg1993}, and also \cite[Section 3.3.2]{AlvarezLefeuvreLoweSmith} for more general elliptic curvature functions)
\begin{equation}
\label{eqn:EqnOfJacobiOperator}
Ju = -(\overlineRic_{\nu\nu}+\|A\|^2)u - \Delta^gu\ ,
\end{equation}
where here $\|A\|^2=\tr(A^2)$ and $\Delta^gu=\tr_g(\nabla^2u)$ denotes the standard Laplace--Beltrami operator of $g$.

The Jacobi operator is understood geometrically as the second variation of area. Indeed, denoting by $\mathcal{A}_u$ the area of $e_u$ for all $u$
\begin{equation*}
\partial_t^2 \mathcal{A}_{tu}|_{t=0}=\int_D u\,Ju\,~\dd\opArea_g + \int_{\partial D}u\partial_n u~ \dd \ell\ ,
\end{equation*}
where $\dd\ell$ is the arc-length measure, $\partial_n u$ here denotes the derivative of $u$ in the outward-pointing unit conormal direction over $\partial D$ (see, for example, \cite{ColdingMinicozzi2011,Rosenberg1993}). In particular, when $u$ vanishes over $\partial D$, this reduces to
\begin{equation*}
\partial_t^2 \mathcal{A}_{tu}|_{t=0}=\int_D u\,Ju\,\dd\opArea_g\ .
\end{equation*}

\subsubsection{Dirichlet spectrum and strict stability} Let $L$ be a second-order, self-adjoint, linear, elliptic partial differential operator over $D$. We define a \emph{Dirichlet eigenfunction} of $L$ with \emph{eigenvalue} $\lambda\in\C$ to be a solution $u$ of the boundary value problem
\begin{equation*}
Lu=\lambda u\qquad\text{and}\qquad u|_{\partial D} = 0\ .
\end{equation*}
The \emph{Dirichlet spectrum} of $L$ is then defined to be the set of all Dirichlet eigenvalues. Since $L$ is self-adjoint, this spectrum is a discrete subset of $\mathbb{R}$ which only accumulates at $\infty$ (see \cite[Section 8]{GilbargTrudinger2001}).

We denote the \emph{least Dirichlet eigenvalue of} $L$ by $\lambda_0(L)$, and we say that $L$ is \emph{strictly stable} whenever
\begin{equation*}
\lambda_0(L)>0\ .
\end{equation*}
We say that an immersed minimal disk is \emph{strictly stable} whenever its Jacobi operator has this property, that is, whenever $\lambda_0(J)>0$. Geometrically, this condition implies that every normal perturbation of $D$ which preserves the boundary increases the area of $D$ non-trivially at second order.

\subsection{Geometric measure theory and stationary varifolds}\label{ss.gmt_varifouille}

Our proof of uniqueness uses the theory of stationary varifolds, which generalize smooth minimal surfaces. We refer the reader to \cite[Chapter 8]{simon1983lectures} for the general theory, and to \cite{white2010maximum} for the case of manifolds with boundary.

In the present context, a \emph{varifold} over $\closedthreeball$ is a Radon measure on the unit tangent bundle $\opU\closedthreeball$, where a unit
vector $N\in\opU_x\closedthreeball$ is understood to represent the oriented $2$-plane $N^\perp\subseteq T_x\closedthreeball$. The \emph{mass} of a varifold $V$ is defined by
\begin{equation*}
M(V):=V(\opU\closedthreeball)\ .
\end{equation*}

Given an oriented, immersed surface $S\subseteq\closedthreeball$, with unit normal $N:S\rightarrow\opU\closedthreeball$, its \emph{associated varifold} $V:=V(S)$ is defined to be the push-forward through $N$ of the area measure of this surface, that is, for every open subset
$\Omega$ of $\opU\closedthreeball$,
\begin{equation*}
V(\Omega) := \opArea(N^{-1}(\Omega))\ .
\end{equation*}
In this case, the mass of $V$ coincides with the area of $S$.

In order to study variations of varifolds, it is convenient to view $\closedthreeball$ as a subset of a larger, open manifold $\Omega$, say. The following definitions are trivially independent of the extension chosen. Given a compactly supported vector field $X$ over $\Omega$, we denote its flow by $(\phi_t)_{t\in\mathbb{R}}$, and we define the \emph{first variation} of a varifold
$V$ in the direction of $X$ by 
\begin{equation*}
\delta V(X) := \partial_t M(\phi_{t*}V)|_{t=0}\ ,
\end{equation*}
where $\phi_{t*}V$ here denotes the push-forward of $V$ through $\phi_t$. We say that a varifold $V$ over $\closedthreeball$ is \emph{stationary} whenever $\delta V(X)$ vanishes for every vector field $X$.

By the first variation formula (see Section $39$ of \cite{simon1983lectures}),
\begin{equation*}
\delta V(X) = \int_{\opU\closedthreeball}\nabla^{P}\!\cdot X~\dd V\ ,
\end{equation*}
where $P:=N^\perp$ and $\nabla^P\!\cdot X$ denotes the divergence of $X$ along the plane $P$. When $V$ is the varifold associated to an immersed disk $D$, say, Stokes theorem yields, for every vector field $X$,
\begin{equation}
\label{eqn:FirstVariationOfDisk}
\delta V(X)= 2\int_D H\langle X,\nu\rangle~\dd\opArea
+ \int_{\partial D}\langle X,n\rangle~\dd\ell\ ,
\end{equation}
where here $\nabla^D\!\cdot X$ denotes the divergence of $X$ along $TD$, $\nu$ denotes the unit normal vector field over $D$ with respect to which its shape operator is defined, and $n$ denotes the outward-pointing unit conormal vector field over $\partial D$.

Finally, we will use the following consequence of White's maximum principle for stationary varifolds (see \cite[Theorem 1]{white2010maximum}).
\begin{theorem}[White's geometric maximum principle]
\label{thm:White}
If the mean curvature of $\mathbb{S}^2$ is everywhere positive, then $\mathbb{S}^2$ is disjoint from the support of every
stationary varifold over $(\closedthreeball,\overline g)$.
\end{theorem}
\begin{remark} We note that stability in White's theorem is expressed in \cite{white2010maximum} in terms of variations with respect to vector fields which are admissible in the sense that they point inward from the boundary of $X$. By linearity of the first variation, the condition of admissability is trivially bypassed by working in the extension $\Omega$.
\end{remark}

\section{Conformal change and intrinsic estimates on minimal disks}

Let $D$ be an immersed minimal disk in $(\closedthreeball,\overline g)$ bounded by a closed curve $c$ of total curvature at most $4\pi$. Recall that $g$ denotes the induced metric on $D$, $\nabla$ denotes its Levi--Civita connection, and $\Delta=\Delta^g$ denotes its Laplace--Beltrami operator. We also denote by $\Delta=-\nabla^\ast\nabla$ the rough Laplacian  acting on tensor fields.

\subsection{Conformal change of metric}

Following the original idea of Nitsche and Barbosa--do Carmo, we introduce the function
\begin{equation}
\label{eqn:NitscheDefinitionOfPhi}
\phi:=1+\frac12\|A\|^2\ ,
\end{equation}
and the conformally equivalent metric
\begin{equation}
\label{eqn:DefinitionOfH}
h:=\phi g\ .
\end{equation}

The key point is that, in $2$-dimensions, the Laplace--Beltrami operator satisfies the conformal transformation law
\begin{equation}
\label{eq:conformal-laplacian}
\Delta^h = \frac{1}{\phi}\Delta^g\ ,
\end{equation}
so that
\begin{equation}
Ju=(2-\overlineRic_{\nu\nu})\,u-\phi(\Delta^h+2) u\ .
\end{equation}

Thus, when the ambient space has negative Ricci curvature, proving stability of the Jacobi operator reduces to proving a spectral estimate for the Laplace--Beltrami operator of the conformal metric $h$. This estimate is in turn obtained via a simple comparison lemma (Proposition~\ref{proposition:ComparisonLemma}) once the following geometric bounds have been established.
\begin{itemize}
\item $\sigma_h\leq 1$,\ \text{and}
\item $\operatorname{Area}_h(D)\leq 2\pi$.
\end{itemize}

In our present framework, by \cite[Theorem 1.159]{Besse1987} with $h=e^{2\psi}g$ and $\psi:\log(\phi)/2$, the curvature of $h$ is given by
\begin{equation}\label{eq:conformal-curvature}
 \sigma_h=\frac1\phi\left(\sigma-\frac12\Delta^g\log\phi\right)\ .
\end{equation}
Estimating the curvature of $h$ thus reduces to controlling $\Delta^g\log\phi$. However, by Leibniz's rule,
\begin{equation}\label{eq:leibniz}
\tfrac12\Delta^g\|A\|^2=\langle\Delta^gA,A\rangle+\|\nabla A\|^2\ ,
\end{equation}
The main challenge is thus to compute the rough Laplacian of the second fundamental form. To this end, we use the Simons identity (see \cite[Theorem 4.2.1]{Simons68}) adapted to the case where the ambient manifold has non-constant curvature.

\subsection{Simons identity for the second fundamental form}

The Simons identity provides a formula for the Laplacian $\Delta A$ of the second fundamental form along each principal direction.
\begin{lemma}[Simons identity]\label{lem:Simons}
Let $D$ be a minimal disk. At a point $p$, in an adapted frame,
\begin{equation}\label{eq:Simons-diag}
(\Delta A)_{ii}
=\lambda_i\bigl(2\sigma+2\overline \sigma-\overlineRic_{\nu \nu}\bigr)
+(\overline\nabla\overlineRic)_{i\nu i}-(\overline\nabla\overlineRic)_{j\nu j}\ ,
\end{equation}
where $j\neq i$.
\end{lemma}

\begin{proof}
By definition of the rough Laplacian, in any adapted frame at $p$,
\begin{equation*}
(\Delta A)(p)_{ii}=\sum_{k=1}^2(\nabla_k\nabla_k A_{ii})(p)\ ,
\end{equation*}
where here and in what follows we use the notation of \eqref{eqn:CovariantDerivativeTerminology}. We now work only at $p$, and, for ease of terminology, we suppress $p$ in what follows.

For all $a$, $b$, $c$, denote
\begin{equation*}
\tilde{R}_{abc} := \overlineR_{ab\nu c}\ .
\end{equation*}
Now fix indices $i$ and $k$. By the Codazzi--Mainardi equation \eqref{Codazzi_Mainardi1},
\begin{equation*}
\begin{aligned}
\nabla_k\nabla_k A_{ii} &= \nabla_k\big(\nabla_i A_{ki} + \tilde{R}_{kii}\big)\\
&= \nabla_i\nabla_k A_{ki} + \big(\nabla_k\nabla_i - \nabla_i\nabla_k)A_{ki} + \nabla_k\tilde{R}_{kii}\ .
\end{aligned}
\end{equation*}
Applying the Codazzi--Mainardi equation \eqref{Codazzi_Mainardi1} again yields
\begin{equation*}
\nabla_k\nabla_k A_{ii} = \underbrace{\nabla_i\nabla_i A_{kk}}_{(\mathrm{I})} + \underbrace{\nabla_i\tilde{R}_{kik}}_{(\mathrm{II})} +
\underbrace{\big(\nabla_k\nabla_i - \nabla_i\nabla_k\big)A_{ki}}_{(\mathrm{III})} + \underbrace{\nabla_k\tilde{R}_{kii}}_{(\mathrm{IV})}\ .
\end{equation*}
We now address each of the above terms separately.\\

\noindent\emph{Term $(\mathrm I)$:} Since $D$ is minimal,
\begin{equation*}
\sum_{k=1}^2\nabla_i\nabla_i A_{kk} = 2\nabla_i\nabla_i H = 0\ .
\end{equation*}

\noindent\emph{Terms $(\mathrm{II})$ and $(\mathrm{IV})$:} Note that $(\mathrm{IV})$ is obtained from $(\mathrm{II})$ by inverting the roles of $i$ and $k$ and multiplying by $(-1)$. In particular, their sum vanishes unless $i\neq k$. Suppose now that $i\neq k$. By Lemma \ref{lem:star}, Item $(3)$, near $p$,
\begin{equation*}
\tilde{R}_{kik} = \overlineR_{ki\nu k} = \overlineRic_{i\nu}\ .
\end{equation*}
Since the frame is adapted at $p$, it follows that
\begin{equation*}
(\mathrm{II}) = \overline{\nabla}_i\overlineRic_{i\nu} - \lambda_i\overlineRic_{\nu\nu} + \lambda_i\overlineRic_{ii}\ .
\end{equation*}
Thus
\begin{equation*}
(\mathrm{II}) + (\mathrm{IV}) = \big(\overline{\nabla}_i\overlineRic_{i\nu} - \overline{\nabla}_k\overlineRic_{k\nu}\big)
-(\lambda_i-\lambda_k)\overlineRic_{\nu\nu} + \lambda_i\overlineRic_{ii} - \lambda_k\overlineRic_{kk}\ .
\end{equation*}
By minimality, $\lambda_i=-\lambda_k$. Thus, since $\{i,k\}=\{1,2\}$,
\begin{equation*}
(\mathrm{II}) + (\mathrm{IV}) = \big(\overline{\nabla}_i\overlineRic_{i\nu} - \overline{\nabla}_k\overlineRic_{k\nu}\big) - 2\lambda_i\overlineRic_{\nu\nu} + \lambda_i\big(\overlineRic_{11} + \overlineRic_{22}\big)\ .
\end{equation*}
Hence, by Lemma \ref{lem:star}, Items $(1)$ and $(2)$,
\begin{equation*}
(\mathrm{II}) + (\mathrm{IV}) = \big(\overline{\nabla}_i\overlineRic_{i\nu} - \overline{\nabla}_k\overlineRic_{k\nu}\big) + \lambda_i(2\overline{\sigma}-\overlineRic_{\nu\nu})\ .
\end{equation*}
Thus, for general $i$ and $k$,
\begin{equation*}
(\mathrm{II}) + (\mathrm{IV}) = \big(\overline{\nabla}_i\overlineRic_{i\nu} - \overline{\nabla}_k\overlineRic_{k\nu}\big) + \lambda_i(2\overline{\sigma}-\overlineRic_{\nu\nu})(1-\delta_{ik})\ .
\end{equation*}

\noindent\emph{Term $(\mathrm{III})$:} Since the frame is adapted at $p$, the Ricci identity \eqref{eq.ricci_identity} yields, for all $i$, $k$,
\begin{equation*}
(\mathrm{III}) = \nabla_k\nabla_iA_{ki} - \nabla_i\nabla_k A_{ik} = -A(R(e_k,e_i)e_k,e_i) - A(e_k,R(e_k,e_i)e_i)\ .
\end{equation*}
Thus, by \eqref{eq.Riemann_tensor_dim2},
\begin{equation*}
(\mathrm{III}) = \sigma(\lambda_i-\lambda_k)\ .
\end{equation*}
By minimality, when $i\neq k$, $\lambda_i=-\lambda_k$, so that
\begin{equation*}
(\mathrm{III}) = 2\lambda_i\sigma(1-\delta_{ik})\ .
\end{equation*}

The result now follows upon summing these terms over all $k$.
\end{proof}

\begin{corollary}\label{cor:Simons-scalar}
Let $D\subset\closedthreeball$ be a minimal disk. Then, in an adapted frame at $p$,
\begin{equation}\label{eq:Simons-scalar}
\frac12\Delta\|A\|^2=
\|\nabla A\|^2+
\bigl(2\sigma+2\overline\sigma-\overlineRic_{\nu\nu}\bigr)\|A\|^2+
2\,\sum_i\lambda_i(\overline\nabla\overlineRic)_{i\nu i}.
\end{equation}
\end{corollary}

\begin{proof}
Indeed, by Leibniz's formula,
\begin{equation*}
\frac{1}{2}\Delta\|A\|^2 = \|\nabla A\|^2 + \langle\Delta A,A\rangle\ .
\end{equation*}
The matrix $A$ is diagonal at $p$ and, furthermore, by minimality, $\lambda_1+\lambda_2=0$. Thus, bearing in mind Simons' formula \eqref{eq:Simons-diag},
\begin{eqnarray*}
\frac12\Delta\|A\|^2
&=&\|\nabla A\|^2+\sum_i\lambda_i\Delta A_{ii}\\
&=&\|\nabla A\|^2+\sum_i\lambda_i^2(2\sigma+2\bar \sigma-\overlineRic_{\nu\nu})+\sum_i\lambda_i(\overline\nabla\overlineRic)_{i\nu i}-\sum_j(-\lambda_j)(\overline\nabla\overlineRic)_{j\nu j}\\
&=&\|\nabla A\|^2+(2\sigma+2\bar \sigma-\overlineRic_{\nu\nu})\|A\|^2+2\sum_i\lambda_i(\overline\nabla\overlineRic)_{i\nu i}\ ,
\end{eqnarray*}
as desired.
\end{proof}

As a reality check, we may verify our this formula in the case where the ambient manifold has constant sectional
curvature $-1$. Here,
\begin{equation*}
\overline\sigma=-1\ ,\qquad
\overlineRic_{\nu\nu}=-2\ ,\qquad\text{and}\qquad
\overline\nabla\overlineRic=0\ .
\end{equation*}
Corollary~\ref{cor:Simons-scalar} thus reduces to
\begin{equation*}
\frac{1}{2}\Delta\|A\|^2
=\|\nabla A\|^2+2\sigma\|A\|^2\ .
\end{equation*}
By Gauss' equation and minimality,
\begin{equation*}
\sigma=\overline\sigma+\kappa_{\mathrm{ext}}=-1-\frac12\|A\|^2\ .
\end{equation*}
Thus
\begin{equation*}
\frac{1}{2}\Delta\|A\|^2=
\|\nabla A\|^2-\left(2+\|A\|^2\right)\|A\|^2\ ,
\end{equation*}
that is,
\begin{equation*}
\frac{1}{2}\Delta\|A\|^2=\|\nabla A\|^2-\|A\|^4-2\|A\|^2\ ,
\end{equation*}
which is the precisely the classical Simons identity for minimal surfaces in hyperbolic $3$--space: \cite{Simons68}.

\subsection{Formula for the curvature of the conformal metric}
We can now compute the curvature of the metric $h=\phi g$ where we recall $\phi=1+\frac12\|A\|^2$.

\begin{lemma}[Curvature of the conformal metric]\label{lem:exact-Kh}
The metric $h=\phi g$ satisfies
\begin{equation}
\label{eq:exact-Kh}
\sigma_h=
\frac{\sigma}{\phi}-\frac{\|A\|^2}{2\phi^2}\big(2\sigma+2\overline\sigma-\overlineRic_{\nu\nu}\big)
-\frac{1}{\phi^2}\sum_i\lambda_i(\overline\nabla\overlineRic)_{i\nu i}
-\frac{\|\nabla A\|^2}{2\phi^2}
+\frac{\left|\dd\|A\|^2\right|^2}{8\phi^3}\ .
\end{equation}
\end{lemma}

\begin{proof}
Indeed, since $\phi=1+\frac12\|A\|^2$,
\begin{equation*}
\Delta\log\phi
=
\frac{1}{2\phi}\Delta\|A\|^2-\frac{1}{4\phi^2}\left|\dd\|A\|^2\right|^2\ .
\end{equation*}
Thus, by \eqref{eq:conformal-curvature},
\begin{eqnarray*}
 \sigma_h
 &=&\frac1\phi\left(\sigma-\frac12\Delta\log\phi\right) \\
 &=&\frac{\sigma}{\phi}
   -\frac{1}{4\phi^2}\Delta\|A\|^2
   +\frac{1}{8\phi^3}\left|\dd\|A\|^2\right|^2\\
 &=&  \frac{\sigma}{\phi}
 -
 \frac{\|A\|^2}{2\phi^2}
 \bigl(2\sigma+2\overline\sigma-\overlineRic_{\nu\nu}\bigr)
 -
 \frac{1}{\phi^2}\sum_i\lambda_i(\overline\nabla\overlineRic)_{i\nu i}   -
 \frac{\|\nabla A\|^2}{2\phi^2}
 +
 \frac{\left|\dd\|A\|^2\right|^2}{8\phi^3}\ ,
\end{eqnarray*}
as desired.
\end{proof}

\subsection{A refined Kato inequality} It remains only to control the contribution of the term $\left|d\|A\|^2\right|^2$. To this end, we prove a refined version of the Kato inequality. Recall the standard Kato inequality for tensors (see \cite{Kato,Calderbank_Gauduchon_Herzlich})
\begin{equation*}
|\dd\|T\||\leq\|\nabla T\|\ .
\end{equation*}
In the case of the second fundamental form, this yields
\begin{equation*}
|\dd\|A\|^2|^2\leq 4\|A\|^2\|\nabla A\|^2\ .
\end{equation*}
Our refinement of the Kato inequality capitalizes on the Codazzi constraint. Indeed, the defect of symmetry of $\nabla A$ is precisely by the mixed Ricci covector $\overlineRic(\nu,\cdot)|_{TD}$ which is of zero'th order, and may thus be controlled.

\begin{lemma}[Refined Kato inequality]\label{lem.refined_Kato}
For every $r>0$,
\begin{equation}\label{eq:Kato-general}
\left|\dd\|A\|^2\right|^2
\leq
2(1+r)\|A\|^2\|\nabla A\|^2
+4(1+r^{-1})\|A\|^2\sum_i|\overlineRic_{\nu i}|^2\ .
\end{equation}
In particular, with $\phi=1+\frac12\|A\|^2$ and $r=\frac{2}{\|A\|^2}$,
\begin{equation}\label{eq:Kato-phi}
-\frac{\|\nabla A\|^2}{2\phi^2}
+\frac{\left|\dd\|A\|^2\right|^2}{8\phi^3}
\leq
\frac{\|A\|^2}{2\phi^2}\sum_i|\overlineRic_{\nu i}|^2\ .
\end{equation}
\end{lemma}

This yields the following upper bound for the curvature of $h$.
\begin{corollary}\label{coro_upper_bound_sigmah}
The curvature of the metric $h=\phi g$ satisfies
\begin{equation}
 \sigma_h
 \leq
 \frac{\sigma}{\phi}
 - \frac{\|A\|^2}{2\phi^2}
 \bigl(2\sigma+2\overline\sigma-\overlineRic_{\nu\nu}\bigr)
 - \frac{1}{\phi^2}\sum_i\lambda_i(\overline\nabla\overlineRic)_{i\nu i}+ \frac{\|A\|^2}{2\phi^2}\sum_i|\overlineRic_{\nu i}|^2\ .
\end{equation}
\end{corollary}

\begin{proof}[Proof of Lemma \ref{lem.refined_Kato}]
Indeed, define $x:=(x_1,x_2)$, $y:=(y_1,y_2)$, and $z:=(z_1,z_2)$ by
\begin{equation*}
x_i:=\nabla_i A_{11}\ ,\qquad y_i:=\nabla_i A_{12}\ ,\qquad z_1:=\overlineRic_{\nu 1}\ ,\quad\text{and}\qquad z_2:=-\overlineRic_{\nu 2}\ .
\end{equation*}

We first show that
\begin{equation}
\label{eqn:KatoProofI}
x=Jy+z\ ,
\end{equation}
where $J$ here denotes rotation by $\pi/2$, that is
\begin{equation*}
J(y_1,y_2) = (-y_2,y_1)\ .
\end{equation*}
Indeed, by the Codazzi--Mainardi equations, and the fact that $A$ is symmetric and trace-free,
\begin{equation*}
y_2+x_1=\nabla_2 A_{12}+\nabla_1 A_{11}=\nabla_2 A_{12}-\nabla_1 A_{22}=\overline{\opR}_{21\nu2}=\overlineRic_{\nu 1}=z_1\ ,
\end{equation*}
and
\begin{equation*}
x_2-y_1=\nabla_2A_{11}-\nabla_1 A_{21}=\overline{\opR}_{21\nu1}=-\overlineRic_{\nu 2}=z_2\ ,
\end{equation*}
and \eqref{eqn:KatoProofI} follows.

We now claim that
\begin{equation}
\label{eqn:KatoProofII}
\|\nabla A\|^2 = 2\big(\|x\|^2 + \|y\|^2\big)\ .
\end{equation}
Indeed, since $A$ is symmetric and trace-free, for each $i$,
\begin{equation*}
\nabla_i A_{11}=-\nabla_i A_{22}=x_i\qquad\text{and}\qquad\nabla_i A_{12}=\nabla_i A_{21}=y_i\ .
\end{equation*}
Hence, for each $i$,
\begin{equation*}
|\nabla_i A|^2=2(x_i^2+y_i^2)\ ,
\end{equation*}
and \eqref{eqn:KatoProofII} follows upon summing over all $i$.

Since $A$ is diagonal at $p$, for each $k$,
\begin{equation*}
\dd\|A\|^2(e_k)=2\langle\nabla_{k}A,A\rangle=2\sum_i\lambda_i\nabla_k A_{ii}=4\lambda_1x_k\ .
\end{equation*}
Thus
\begin{equation*}
\left|\dd\|A\|^2\right|^2
=\sum_k(4\lambda_1 x_k)^2
=16\lambda_1^2\|x\|^2
=8\|A\|^2\|x\|^2\ .
\end{equation*}
Finally, by Cauchy--Schwarz and Young's inequality, for all $r>0$,
\begin{equation*}
\|x\|^2=\|Jy+z\|^2 =\langle Jy,Jy \rangle + 2 \langle Jy,z \rangle + \langle z,z \rangle \leq(1+r)\|y\|^2+(1+r^{-1})\|z\|^2\ ,
\end{equation*}
so that
\begin{eqnarray*}
\left|\dd\|A\|^2\right|^2&=&8\|A\|^2|x|^2\\
&=&4\|A\|^2|x|^2+4\|A\|^2|x|^2\\
&\leq &4(1+r)\|A\|^2(|x|^2+|y|^2)+4(1+r^{-1})\|A\|^2|z|^2\\
&=& 2(1+r)\|A\|^2\|\nabla A\|^2+4(1+r^{-1})\|A\|^2\sum_i|\overlineRic_{\nu i}|^2\ ,
\end{eqnarray*}
as desired.
\end{proof}

\subsection{Pinching of sectional curvature and upper bounds on the Ricci curvature}

We now assume that we have the following pinching of the sectional curvatures of $\bar g$
\begin{equation}\label{eq.pinching_sect}
-(1+\alpha)\leq\overline{\sect}\leq-1\ .
\end{equation}
The following useful estimate is specific to the $3$-dimensional case.
\begin{proposition}\label{prop_bound_ricci}
Under the pinching condition \eqref{eq.pinching_sect},
\begin{equation*}
\sum_i|\overlineRic_{\nu i}|^2\leq -(\bar\sigma+1)(\bar\sigma+1+\alpha)\ .
\end{equation*}
\end{proposition}
\noindent Proposition \ref{prop_bound_ricci} follows from the spectral properties of the curvature operator in dimension $3$.

\subsubsection{The curvature operator}We first review some basic linear algebra. Let $E$ be a vector space, and let $\Lambda^2 E$ denote the space of antisymmetric bivectors over $E$. This may be viewed as the dual space of the more familiar space $\Lambda^2 E^*$ of $2$-forms over $E$. Elements of $\Lambda^2 E$ are bivectors of the form $X\wedge Y$, where
\begin{equation*}
X\wedge Y = - Y\wedge X\ .
\end{equation*}
Every inner product on $E$ induces a unique inner product on $\Lambda^2 E$ such that, for any orthonormal pair $(X,Y)$ of vectors in $E$,
\begin{equation*}
\|X\wedge Y\|^2 = 1\ .
\end{equation*}
In particular, if $(e_1,\cdots,e_d)$ is an orthonormal basis of $E$, then $(e_i\wedge e_j)_{i<j}$ is an orthonormal basis of $\Lambda^2E$.

We will use the following fact, specific to the $3$-dimensional case.
\begin{lemma}
\label{lemma:decomposability}
Let $E$ be a $3$-dimensional vector space. Every element $\omega$ of $\Lambda^2 E$ is \emph{decomposable} in the sense that it may be written in the form $\omega = X\wedge Y$ for some pair of vectors $X,Y\in E$.
\end{lemma}

\begin{proof}
We furnish $E$ with an inner product and orientation. In $3$-dimensions, the Hodge star operator defines an isomorphism $*:E\rightarrow\Lambda^2 E$ such that, for all $X,Y\in E$,
\begin{equation*}
X\wedge Y = *(X\times Y)\ ,
\end{equation*}
where $\times$ here denotes the classical vector product. However, every vector $Z\in E$ can be written in the form $Z=X\times Y$. Indeed, it suffices to choose $X$ and $Y$ non-colinear, orthogonal to $Z$, and bounding a parallelogram of area equal to $\|Z\|$. The result follows.
\end{proof}

We denote by $\Lambda^2\opT\closedthreeball$ the bivector bundle over $\closedthreeball$, and we define the \emph{curvature operator} to be the linear endomorphism $\mathcal{R}:\Lambda^2T\closedthreeball\to\Lambda^2T\closedthreeball$ given by
\begin{equation*}
\langle\mathcal{R}(X\wedge Y),Z\wedge W\rangle=\overlineR(X,Y,W,Z)\ .
\end{equation*}
where $\langle\cdot,\cdot\rangle$ denotes the metric in $\Lambda^2\opT\closedthreeball$ induced by $\bar g$. By the standard symmetries of the Riemann curvature tensor, $\mathcal{R}$ is self-adjoint. Furthermore, for any orthonormal pair $(X,Y)$ of tangent vectors,
\begin{equation*}
\langle\mathcal{R}(X\wedge Y),X\wedge Y\rangle=\overline{\sect}(X,Y)\ .
\end{equation*}
By Lemma \ref{lemma:decomposability}, in the $3$-dimensional case, every element of $\Lambda^2\opT\closedthreeball$ is decomposable, it follows that the eigenvalues of $\mathcal{R}$ are sectional curvatures. The pinching condition \eqref{eq.pinching_sect} therefore yields
\begin{equation}
\label{eq.spectr_curv}
\Spec(\mathcal{R})\subset [-(1+\alpha),-1]\ .
\end{equation}

\subsubsection{Spectral centering} Trivially, if $A$ is a self-adjoint operator with $\Spec(A)\subset[a,b]$, then
\begin{equation*}
\left\|A-\frac{(a+b)}{2}\Id\right\|\leq\frac{b-a}{2}\ ,
\end{equation*}
where the norm here is the operator norm. Upon setting $a=-(1+\alpha)$ and $b=-1$, \eqref{eq.spectr_curv} therefore yields
\begin{equation}\label{eq.upper_bound_curv}
\left\|\mathcal{R}+\left(1+\frac{\alpha}{2}\right)\Id\right\|\leq\frac{\alpha}{2}\ .
\end{equation}

\begin{proof}[Proof of Proposition \ref{prop_bound_ricci}]
Let $(e_1,e_2,\nu)$ be an adapted frame of $\closedthreeball$ near $p$ and define the orthonormal frame of $\Lambda^2T\closedthreeball$ by
\begin{equation*}
\omega_1:=e_2\wedge\nu\ ,\quad\omega_2:=\nu\wedge e_1\ ,\qquad\text{and}\qquad\omega_3=e_1\wedge e_2\ .
\end{equation*}
Note that
\begin{equation*}
\langle\mathcal{R}\omega_3,\omega_1\rangle = \overlineR_{12\nu2} = -\overlineRic_{1\nu}\ ,
\end{equation*}
and
\begin{equation*}
\langle\mathcal{R}\omega_3,\omega_2\rangle = \overlineR_{121\nu} = -\overlineRic_{2\nu}\ .
\end{equation*}

Denote
\begin{equation*}
\mathcal{R}' := \mathcal{R}+\left(1+\frac{\alpha}{2}\right)\Id\ ,
\end{equation*}
Since $\mathcal{R}'$ and $\mathcal{R}$ have the same off-diagonal coefficients, and since by definition $\langle\mathcal{R}\omega_3,\omega_3\rangle=\bar\sigma$, we obtain
\begin{eqnarray*}
\mathcal{R}'\omega_3 &=& -\overlineRic_{1\nu}\omega_1-\overlineRic_{2\nu}\omega_2 + \langle\mathcal{R}'\omega_3,\omega_3\rangle\omega_3\\
                     &=& -\overlineRic_{1\nu}\omega_1-\overlineRic_{2\nu}\omega_2+\left(\bar\sigma+1+\frac{\alpha}{2}\right)\omega_3\ .
\end{eqnarray*}
Thus, by \eqref{eq.upper_bound_curv},
\begin{equation*}
\sum_i|\overlineRic_{\nu i}|^2+\left(\bar\sigma+1+\frac{\alpha}{2}\right)^2\leq\|\mathcal{R}'\omega_3\|^2\leq\frac{\alpha^2}{4}\ ,
\end{equation*}
which yields
$$\sum_i|\overlineRic_{\nu i}|^2\leq\frac{\alpha^2}{4}-\left(\bar\sigma+1+\frac{\alpha}{2}\right)^2=(-\bar\sigma-1)(\bar \sigma+1+\alpha)\ ,$$
as desired.
\end{proof}

\begin{remark}
This estimate is specific to the $3$-dimensional case. In higher dimensions, eigenvectors of the curvature operator need not be decomposable, the spectrum of $\mathcal R$ is thus no longer a subset of the set of sectional curvatures, and the pinching assumption therefore no longer implies the spectral estimate \eqref{eq.upper_bound_curv}.
\end{remark}

\subsection{Optimizing the pinching estimate}
We now turn to the problem of finding explicit bounds on $\alpha>0$ that imply that $\sigma_h\leq 1$ whenever
\begin{equation}\label{eq.pinching_hypotheses}-(1+\alpha)\leq\overline{\sect}\leq -1\ ,\qquad\text{and}\qquad\left\|\overline{\nabla}\overlineRic\right\|\leq\alpha\ .
\end{equation}

Combining Corollary \ref{coro_upper_bound_sigmah} and Proposition \ref{prop_bound_ricci} yields
\begin{equation}\label{eq.borne_sigma_h_1}
\sigma_h\leq \frac{\sigma}{\phi}-\frac{\|A\|^2}{2\phi^2}(2\sigma+2\bar\sigma-\overlineRic_{\nu\nu})-\frac{1}{\phi^2}\sum_i\lambda_i(\overline\nabla\overlineRic)_{i\nu i}-\frac{\|A\|^2}{2\phi^2}(\bar\sigma+1)(\bar\sigma+1+\alpha)\ .
\end{equation}
where, we recall,
\begin{equation*}
\phi=1+\frac12\|A\|^2\ .
\end{equation*}
We now introduce the functions
\begin{equation*}
s(p):=|\bar\sigma(p)|=-\bar\sigma(p)\quad\text{and}\quad\delta(p):=|\det(A(p))|=-\det(A(p))\ .
\end{equation*}
Note that, by minimality,
\begin{equation*}
\|A\|^2=2\delta\qquad\text{and}\qquad\phi=1+\delta\ .
\end{equation*}
We now estimate the terms appearing in \eqref{eq.borne_sigma_h_1} in terms of the functions $\phi$, $\delta$, and $s$.

\begin{enumerate}
\item By Gauss' equation \eqref{GaussEqn2},
\begin{equation*}
\sigma=-(s+\delta).
\end{equation*}
\item By \eqref{lem:star},
\begin{equation*}
2\bar\sigma-\overlineRic_{\nu \nu}=(\bar\sigma-\overline{\sect}_{1\nu})+(\bar\sigma-\overline{\sect}_{2\nu}).
\end{equation*}
Thus, since $\overline{\sect}_{2\nu},\overline{\sect}_{1\nu}\leq -1$,
\begin{equation*}
2\bar\sigma-\overlineRic_{\nu \nu}\geq 2(1-s)\ ,
\end{equation*}
and hence
\begin{equation*}
-\left(2\bar\sigma-\overlineRic_{\nu \nu}\right)\leq 2(s-1)\ .
\end{equation*}
\item By minimality, for each $i$,
\begin{equation*}
\left|\lambda_i\right|=\sqrt{\delta}\ .
\end{equation*}
Moreover by symmetry of the Ricci tensor $\overlineRic$, the covariant derivative $\overline\nabla\overlineRic$ is symmetric in the last two entries so
$$(\overline\nabla\overlineRic)_{i\nu i}=(\overline\nabla\overlineRic)_{ii\nu} \m{.}$$
This implies, by definition of the Hilbert--Schmidt norm,
$$2\sum_i((\overline\nabla\overlineRic)_{i\nu i})^2=\sum_i((\overline\nabla\overlineRic)_{i\nu i})^2+\sum_i((\overline\nabla\overlineRic)_{ii\nu})^2\leq\|\overline\nabla\overlineRic\|^2\leq\alpha^2\ ,$$

By minimality we have $\lambda_1^2+\lambda_2^2=2\delta$, so by Cauchy--Schwarz,
\begin{equation*}
\left|\sum_i\lambda_i(\overline\nabla\overlineRic)_{i\nu i}\right|\leq\frac{\sqrt{2\delta}\alpha}{\sqrt{2}}=\sqrt{\delta}{\alpha}\ .
\end{equation*}
\end{enumerate}

This yields
\begin{equation}
\label{eq.borne_sdelta}\sigma_h\leq -\frac{s+\delta}{\phi}+\frac{2\delta(s+\delta)}{\phi^2}+\frac{2\delta(s-1)}{\phi^2}+\frac{\alpha\sqrt{\delta}}{\phi^2}+\frac{\delta}{\phi^2}(s-1)(1+\alpha-s)\ .
\end{equation}
\noindent We now obtain an upper bound on $\sigma_h$ independent of $s$.

\begin{lemma}[Upper bound on the curvature]\label{eq.upper_bound_1-Ndelta}
Set
\begin{equation*}
u:=s-1\in[0,\alpha]\quad\text{and}\quad t:=\sqrt{\delta}\in[0,\infty)\ .
\end{equation*}
Assuming the pinching conditions \eqref{eq.pinching_hypotheses}, we have
\begin{equation*}
\sigma_h\leq 1+\frac{F_\alpha(u,t)}{\phi^2}\ ,
\end{equation*}
where
\begin{equation}\label{eq:Falpha}
F_\alpha(u,t)=\left(-2+(3+\alpha)u-u^2\right)t^2+\alpha t-(2+u).
\end{equation}
In particular, this yields the following criterion.
\begin{equation}\label{eq:criterion}
F_ \alpha\leq 0 \text{ on } [0,\alpha]\times[0,\infty)
 \Longrightarrow
 \sigma_h\leq 1 \text{ on } D\ .
\end{equation}
\end{lemma}

\begin{proof}

Making the changes of variables $u=s-1$, $\delta=t^2$ and $\phi=1+t^2$, we  obtain from \eqref{eq.borne_sdelta} the following

\begin{align*}
\phi^2(\sigma_h-1) &\leq -\phi(s+\delta)+2\delta(s+\delta)+2\delta(s-1)+\alpha\sqrt{\delta}+\delta(s-1)(1+\alpha-s)-\phi^2 \\
&= -(1+t^2)(1+u+t^2)+2t^2(1+u+t^2)\\
&\qquad +2u\,t^2+\alpha t+t^2\,u(\alpha-u)-(1+2t^2+t^4)\\
&= -(1+u)-(2+u)\,t^2-t^4+2(1+u)\,t^2+2t^4\\
&\qquad +2u\,t^2+\alpha t+(\alpha u-u^2)t^2-1-2t^2-t^4\\
&= \left(-2+(3+\alpha)u-u^2\right)t^2+\alpha t-(2+u)\\
&= F_\alpha(u,t)\ ,
\end{align*}
which is the desired inequality.
\end{proof}

We now prove the main result of this section.
\begin{proposition}\label{prop.alphastar}
Let
\begin{equation*}
\alpha_\star:=\frac{4\sqrt{17}-8}{13}=0.653263269...
\end{equation*}
If $0<\alpha\leq\alpha_\star$, then $\sigma_h \leq 1$ over $D$.
\end{proposition}

\begin{proof}
Set
\begin{equation*}
a_\alpha(u):=-2+(3+\alpha)u-u^2\ ,
\end{equation*}
Since $\alpha_\star<2/3$, for $0<\alpha\leq\alpha_\star$ and $u\in[0,\alpha]$,
\begin{equation*}
a_\alpha(u)\leq a_\alpha(\alpha)=-2+3\alpha<0\ .
\end{equation*}
Hence for fixed $u$, the map $F_\alpha(u,t)$ is a concave quadratic polynomial in $t$, with maximum value at
\begin{equation*}
t_\ast(u)=-\frac{\alpha}{2a_\alpha(u)}>0\ .
\end{equation*}
Its maximal value is
\begin{equation*}
F_\alpha(u,t_\ast(u))=-(2+u)-\frac{\alpha^2}{4a_\alpha(u)}\ .
\end{equation*}
This is nonpositive exactly when
\begin{equation}\label{eq.Halpha}
H_\alpha(u):=-4a_\alpha(u)(2+u)-\alpha^2\geq 0\ .
\end{equation}
Differentiating \eqref{eq.Halpha} yields
\begin{equation*}
H_\alpha'(u)=4(3u^2-2(1+\alpha)u-(4+2\alpha))\ .
\end{equation*}
This quadratic polynomial is convex and at the two endpoints of $[0,\alpha]$ it equals
\begin{equation*}
-4(4+2\alpha)<0\qquad\text{and}\qquad 4(\alpha^2-4\alpha-4)<0\ ,
\end{equation*}
respectively. Hence $H_\alpha'(u)<0$ throughout $[0,\alpha]$ and
\begin{equation*}
H_\alpha(u)\geq H_\alpha(\alpha)=16-16\alpha-13\alpha^2\ .
\end{equation*}
The last expression is nonnegative precisely when
\begin{equation*}
0<\alpha\leq\frac{4\sqrt{17}-8}{13}\ .
\end{equation*}
This proves that $F_\alpha\leq 0$ over $[0,\alpha]\times[0,\infty)$ and, by Lemma \ref{eq.upper_bound_1-Ndelta}, that $\sigma_h\leq 1$ over $D$.
\end{proof}

\subsection{Total curvature and area of the conformal metric}
Recall that for an arclength-parametrized smooth curve $c$ in $\closedthreeball$, with unit
tangent $T$, its \emph{curvature vector} and \emph{total curvature} are given respectively by
\begin{equation}\label{eq.total_curvature}
 \boldsymbol\kappa_c:=\overline\nabla_TT\qquad\text{and}\qquad
 K_c:=\int_c|\boldsymbol\kappa_c|\,\dd c\ ,
\end{equation}
where $\dd c$ is the arc-length measure.

\begin{proposition}[Area bound]\label{prop:area}
Denoting $c=\partial D$, if $\overline \sigma\leq-1$, then
\begin{equation*}
\opArea_h(D)\leq K_c-2\pi\ .
\end{equation*}
In particular, if $K_c\leq4\pi$, then $\opArea_h(D)\leq2\pi$.
\end{proposition}

\begin{proof}
By Gauss' formula \eqref{GaussEqn2},
\begin{equation*}
-\sigma = |\overline \sigma|+\delta \geq 1+\delta=\phi\ .
\end{equation*}
Thus, since $h=\phi g$,
\begin{equation*}
\opArea_h(D)
=
 \int_D\phi\,\dd\opArea_g
 \leq
 -\int_D\sigma\,\dd\opArea_g\ .
\end{equation*}

Let $\eta_c$ denote the outward-pointing unit conormal vector field over $c$ in $D$, and let $k_D:= \overline g(\overline\nabla_TT,\eta_{\partial D})$
denote the signed geodesic curvature of $c=\partial D$ viewed as a curve in $D$. By the Gauss--Bonnet theorem,
\begin{equation*}
 -\int_D\sigma\,\dd\opArea_g
 =
 \int_{c}k_D\,\dd s-2\pi\ .
\end{equation*}
Since $|k_D| \leq |\overline\nabla_TT| = |\boldsymbol\kappa_c|$, it follows that
\begin{equation*}
 \opArea_h(D)
 \leq
 K_c-2\pi\ ,
\end{equation*}
as desired.
\end{proof}

\section{Proof of Theorem \ref{thm:Nitsche}}
\label{ss:ProofOfNitsche}

We will make use of the following comparison result.

\begin{proposition}
\label{proposition:ComparisonLemma}
Let $(D,g)$ be a Riemannian disk of Gaussian curvature bounded above by $k_0$, let $(D',g')$ be a geodesic disk in a surface of constant Gaussian curvature equal to $k_0$, let $\mathcal{A}$ and $\mathcal{A}'$ denote their respective areas, and let $\Delta$ and $\Delta'$ denote their respective Laplace operators. If
\begin{equation*}
\mathcal{A}\leq \mathcal{A}'\ ,
\end{equation*}
then
\begin{equation*}
\lambda_0(-\Delta)\geq\lambda_0(-\Delta')\ .
\end{equation*}
\end{proposition}

\begin{remark}
This result follows from \cite[Proposition $3.3$]{BarbosaDoCarmo1980}, where it is stated with the hypothesis $\mathcal{A}=\mathcal{A}'$, together with the domain monotonicity of the smallest Dirichlet eigenvalue.
\end{remark}

\begin{lemma}\label{lem_bottom_hemisphere}
Let $(\Sigma_0,g_0)$ denote the upper unit hemisphere in $\mathbb{R}^3$. Its area $\mathcal{A}^0$ and its Laplace operator $\Delta^0$ satisfy
\begin{equation}
\label{eqn:AreaAndDirichletEigenvalueOfHemisphere}
\mathcal{A}^0 = 2\pi\qquad\text{and}\qquad\lambda_0(-\Delta^0) = 2\ .
\end{equation}
\end{lemma}
\begin{proof}[Proof.] It is well-known that $\mathcal{A}^0=2\pi$. Consider now the function
\begin{equation*}
u_0(x_1,x_2,x_3) := x_3\ .
\end{equation*}
We readily verify that $u_0$ is a Dirichlet eigenfunction of $\Sigma_0$ with eigenvalue $2$. Since $u_0$ does not change sign over the interior of $\Sigma_0$, it is the least Dirichlet eigenfunction, and the result follows.
\end{proof}

\begin{proof}[Proof of Theorem \ref{thm:Nitsche}] Let $D$ be an immersed minimal disk in $(\closedthreeball,\overline{g})$ bounded by $c$. Fix $0<\alpha<\alpha_\star$ and suppose that
\begin{equation}\label{eq:true_hypotheses}
-(1+\alpha)\leq\overline{\sect}\leq -1,\quad\left\|\overline{\nabla}\overlineRic\right\|\leq\alpha\ ,\quad\text{and}\quad K_c\leq 4\pi\ .
\end{equation}
Recall that $g$ denotes the restriction of $\bar g$ to $D$ and that $\Delta^g$ and $J$ denote respectively the Laplace and Jacobi operator of $D$. In addition, $h=\phi g$, where
\begin{equation*}
\phi=1+\frac{1}{2}\|A\|^2\ ,
\end{equation*}
so that $\Delta^g=\phi\Delta^h$ and $J$ may be written in the form
\begin{equation*}
Ju=(2-\overlineRic_{\nu\nu})\,u-\phi(\Delta^h+2) u\ .
\end{equation*}

Since $\overline{\sect}\leq-1$,
\begin{equation*}
\overlineRic_{\nu \nu}\leq -2\ ,
\end{equation*}
so that $Ju>J'u$ in the sense of quadratic forms, where
\begin{equation*}
J'u=-\phi(\Delta^h+2) u\ ,
\end{equation*}
In particular,
\begin{equation*}
\lambda_0(J)>\lambda_0(J')\ .
\end{equation*}

Let $(\Sigma_0,g_0)$ denote the unit hemisphere in $\Bbb{R}^3$, let $\sigma_0=1$ denote its Gaussian curvature, let $\mathcal{A}^0=2\pi$ denote its area, and let $\Delta^0$ denote its Laplace operator. By Lemma \ref{lem_bottom_hemisphere},
\begin{equation*}
\lambda_0(-\Delta^0) = 2\ .
\end{equation*}
Moreover, under the hypotheses \eqref{eq:true_hypotheses}, we have shown that
\begin{enumerate}
\item $\sigma_h\leq 1=\sigma_0$ (Proposition \ref{prop.alphastar}); and
\item $\opArea_h(D)\leq \mathcal{A}^0=2\pi$ (Proposition \ref{prop:area}).
\end{enumerate}
Proposition \ref{proposition:ComparisonLemma} therefore yields
\begin{equation*}
\lambda_0(-\Delta^h)\geq\lambda_0(-\Delta^0)=2\ .
\end{equation*}

Since $\lambda_0(-(\Delta^h+2))>0$, $\dd\opArea_h=\phi\,\dd\opArea_g$, and $\phi\geq 1$, we obtain, for every non zero smooth function $u$ vanishing over the boundary,
\begin{eqnarray*}
\frac{\int_Du\,J'u\,\dd\opArea_g}{\int_Du^2\dd\opArea_g}&=&\frac{\int_Du(-(\Delta^h+2)u)\,\dd\opArea_h}{\int_Du^2\dd\opArea_g}\\
   &\geq & \lambda_0(-(\Delta^h+2))\frac{\int_Du^2\phi\dd\opArea_g}{\int_Du^2\dd\opArea_g}\\
   &\geq & \lambda_0(-(\Delta^h+2))\geq0\ ,\vphantom{\frac{1}{2}}
\end{eqnarray*}
Taking the infimum over all such $u$ yields
\begin{equation*}
0\leq\lambda_0(J')<\lambda_0(J)\ ,
\end{equation*}
so that $D$ is strictly stable, as desired.
\end{proof}

\section{Proof of Theorem \ref{thm:MainResultA}}

\subsection{Minimal surfaces spanned by small circles}
We first address the local problem, that is, we prove existence, local uniqueness, and strict stability of minimal disks bounded by small circles in $(\mathbb{S}^2,\overline{g})$. We achieve this by perturbing the Euclidian case.
\begin{lemma}
Let $c\subseteq\mathbb{R}^2\subseteq\mathbb{R}^3$ be a smooth, simple, closed planar curve. The planar region $S$ bounded by $c$ is the unique minimal surface in $\mathbb{R}^3$ bounded by this curve. Furthermore, $S$ is strictly stable.
\end{lemma}
\begin{proof} $S$ is trivially minimal. To prove uniqueness, let $S'$ be another minimal surface bounded by $c$. By Osserman's Convex Hull Property (see Section \cite[Corollary 1.10]{ColdingMinicozzi2011}), $S'$ is contained in the convex hull of $c$ and, in particular, is contained in $\mathbb{R}^2$, from which uniqueness follows. Finally, since $S$ has vanishing shape operator, its Jacobi operator is simply the Laplace operator $\Delta$. Since the Dirichlet spectrum of $\Delta$ in any bounded planar domain is positive, $S$ is strictly stable, and this completes the proof.
\end{proof}

We express the solution to the local problem as follows.
\begin{lemma}
\label{lemma:MinimalPerturbationExistence}
Let $(c_m)_{m\in\mathbb{N}}$ be a sequence of $C^{2,\alpha}$, simple, closed curves converging in the $C^{2,\alpha}$ sense to the smooth, strictly convex, simple, closed planar curve $c_\infty\subseteq\mathbb{B}^3$. Let $(g_m)_{m\in\mathbb{N}}$ be a sequence of metrics over $4\mathbb{B}^3$ converging in the $C^\infty$ sense to the Euclidian metric. For all sufficiently large $m$, $c_m$ bounds a unique minimal surface $S_m$ in $2\mathbb{B}^3$. Furthermore, for all such $m$, $S_m$ is embedded and strictly stable, and the sequence $(S_m)_{m\geq m_0}$ converges in the $C^{2,\alpha}$ sense to the unique minimal surface bounded by $c_\infty$.
\end{lemma}

Before proving Lemma \ref{lemma:MinimalPerturbationExistence}, we address a couple of technical preliminaries. This first is a standard perturbation result of minimal surface theory. Let $\mathbb{B}^2$ denote the open unit ball in $\mathbb{R}^2$, and let $\Omega\subseteq\mathbb{B}^2$ be an open subset with smooth boundary. Let $C^{2,\alpha}(\Omega)$ and $C^{2,\alpha}(\partial\Omega)$ denote respectively the H\"older spaces of $C^{2,\alpha}$-functions over $\Omega$ and $\partial\Omega$. Note that, if $u$ is an element of either of these two spaces satisfying $\|u\|_{C^0}<1$, then its graph is contained in $2\mathbb{B}^3$. Let $\mathcal{G}^{1,\alpha}:=\mathcal{G}^{1,\alpha}(2\mathbb{B}^3)$ denote the space of $C^{1,\alpha}$ riemannian metrics over $2\mathbb{B}^3$ furnished with the $C^{1,\alpha}$-topology, and let $g_0\in\mathcal{G}^{1,\alpha}$ denote the Euclidian metric.
\begin{lemma}
\label{lemma:ImplicitFunctionTheorem}
There exist $\epsilon,\delta>0$, and a neighbourhood $\mathcal{G}_0$ of $g_0$ in $\mathcal{G}^{1,\alpha}$ such that, for all $g\in\mathcal{G}_0$, and for all $v\in C^{2,\alpha}(\partial\Omega)$ such that $\|v\|_{C^{2,\alpha}}<\epsilon$, there exists a unique $u:=u_{g,v}\in C^{2,\alpha}(\Omega)$ such that
\begin{itemize}
\item[(1)] $\|u\|_{C^{2,\alpha}}<\delta$\ ,
\item[(2)] $u|_{\partial\Omega} = v$\ , and
\item[(3)] the graph of $u$ is minimal with respect to the metric $g$.
\end{itemize}
\end{lemma}
\begin{remark} The functional $(g,v)\mapsto u_{g,v}$ is a smooth functional between Banach manifolds.
\end{remark}
\begin{remark} By elliptic regularity, if $g$ and $v$ are smooth, then so too is $u$.
\end{remark}
\begin{proof}[Sketch of proof.] Define the open subset $\mathcal{U}^{2,\alpha}\subseteq C^{2,\alpha}(\Omega)$ by
\begin{equation*}
\mathcal{U}^{2,\alpha} := \{ u\ |\ \|u\|_{C^0}<1\}\ .
\end{equation*}
Define the functional $H:\mathcal{G}^{1,\alpha}\times\mathcal{U}^{2,\alpha}\rightarrow C^{0,\alpha}(\Omega)$ such that, for all $(g,u)$, and for all $x\in\Omega$, $H(g,u)(x)$ is the mean curvature with respect to $g$ of the graph of $u$ at the point $(x,u(x))$. This functional is smooth as a functional between Banach manifolds. Since $g_0$ is Euclidian, by $(14.102)$ of \cite{GilbargTrudinger2001}, for all $u$,
\begin{equation*}
H(g_0,u) = -\frac{1}{2}\nabla\cdot\bigg(\frac{1}{\sqrt{1+\left|\nabla u\right|^2}}\nabla u\bigg)\ .
\end{equation*}
In particular, the partial derivative of $H$ with respect to the second component at $(g_0,0)$ is
\begin{equation*}
D_2H(g_0,0)\cdot u = -\frac{1}{2}\Delta u\ .
\end{equation*}

Let $R:C^{2,\alpha}(\Omega)\rightarrow C^{2,\alpha}(\partial\Omega)$ denote the restriction operator, and define the functional $\Phi:\mathcal{G}^{1,\alpha}\times\mathcal{U}^{2,\alpha}\rightarrow C^{0,\alpha}(\Omega)\times C^{2,\alpha}(\partial\Omega)$ by
\begin{equation*}
\Phi(g,u) := (H(g,u),R(u))\ .
\end{equation*}
For any $g\in\mathcal{G}^{1,\alpha}$ and $v\in C^{2,\alpha}(\partial\Omega)$, $\Phi(g,u)=(0,v)$ if and only if $u$ is equal to $v$ along the boundary and its graph is minimal with respect to $g$.

$\Phi$ is smooth as a functional between Banach manifolds, and its partial derivative at $(g_0,0)$ with respect to its second component is
\begin{equation*}
D\Phi(g_0,0)\cdot u = L(u) := (-\Delta u/2, R(u))\ .
\end{equation*}
Since $L:C^{2,\alpha}(\Omega)\rightarrow C^{0,\alpha}(\Omega)\times C^{2,\alpha}(\partial\Omega)$ is a linear isomorphism (see, for example, \cite{GilbargTrudinger2001}), the result now follows by the implicit function theorem for smooth functionals between Banach manifolds.
\end{proof}

The second technical preliminary concerns convex subsets of Riemannian manifolds. Let $(X,g)$ be a Riemannian $3$-manifold which is \emph{geodesically convex} in the sense that any two points are connected by a unique geodesic. We say that a real function over $X$ is \emph{convex} whenever its restriction to any constant speed parametrized geodesic is convex. We will say that a closed, convex subset $K\subseteq X$ is $\epsilon$-\emph{convex} whenever every boundary point has a local supporting tangent surface with principal curvatures greater than $\epsilon$. More precisely, $K$ is $\epsilon$-\emph{convex} whenever, for all $p\in\partial K$, and for every supporting normal $N$ of $K$ at $p$, there exists $r>0$, and a smooth, convex function $f:B_r(p)\rightarrow\mathbb{R}$ such that
\begin{itemize}
\item[(1)] $f(p)=0$,
\item[(2)] $\nabla f(p)=N$,
\item[(3)] $K\cap B_r(p)\subseteq f^{-1}(]-\infty,0])$, and
\item[(4)] every level set of $f$ has principal curvatures greater than $\epsilon$.
\end{itemize}
 Let $d_K:X\rightarrow\mathbb{R}$ denote the distance in $X$ to $K$.
\begin{lemma}
\label{lemma:ConvexityOfDistance}
If $K$ is $\epsilon$-convex and if $X$ has sectional curvature bounded above by $C^2>0$, then $d_K$ is convex over $d_K^{-1}(]-\infty,R[)$, where
\begin{equation*}
R = \frac{1}{C}\arctan(\epsilon/C)\ .
\end{equation*}
\end{lemma}
\begin{proof}[Sketch of proof.]We will describe $d_K$ everywhere locally as an envelope of convex functions (that is, locally the sup of a family of convex functions), from which convexity will follow. Choose $q\in X\setminus K$, let $p\in K$ denote the closest point in $K$ to $q$, and let $N$ denote the unit vector at $p$ pointing in towards $q$. Note that, since $p$ is the closest point in $K$ to $q$, $N$ is a supporting normal of $K$ at $p$. Let $f:B_r(p)\rightarrow\mathbb{R}$ be as above, satisfying, in particular $\nabla f(p)=N$, and denote $K':=f^{-1}(]-\infty,0])$, and denote $d:=d_K$ and $d':=d_{K'}$. Then $d(q)=d'(q)$ and, near $q$, $d\geq d'$.

We now use comparison theory (see, for example, \cite{CheegerEbin2008}). Let $Y$ denote the standard sphere in $\mathbb{R}^4$ of radius $1/C$. With $R$ as above, upon comparing with a geodesic sphere of radius $(\pi/2C)-R$ in $Y$, we see that $d'$ is convex near $q$ provided that $d'(q)<R$. The function $d$ is thus everywhere locally an envelope of convex functions, and is therefore itself convex, as desired.
\end{proof}

\begin{proof}[Proof of Lemma \ref{lemma:MinimalPerturbationExistence}.] We first prove existence. Let $S_\infty$ denote the planar region bounded by $c_\infty$. By Lemma \ref{lemma:ImplicitFunctionTheorem}, we may suppose that for all $m$, $c_m$ bounds a minimal surface $S_m$ and that the sequence $(S_m)_{m\in\mathbb{N}}$ converges in the $C^{2,\alpha}$ sense to $S_\infty$. Since $S_\infty$ is strictly stable, and since strict stability is preserved by small perturbations, we may also suppose that $S_m$ is strictly stable for all $m$. By elliptic regularity, we may also suppose that $(S_m)_{m\in\mathbb{N}}$ converges to $S_\infty$ in the $C^\infty$ sense.

In order to prove uniqueness, we use the following construction. Suppose that $c_\infty$ is contained in the $x-y$ plane, and let $\partial_z$ denote the unit vector in the $z$-direction. For all $m$, let $c'_m$ denote the projection of $c_m$ along the $z$-axis onto the $x-y$ plane, and define the cylinder $C_m$ by
\begin{equation*}
C_m := \{ (x,y,z)\ |\ (x,y)\in\overlineint(c_m')\ \text{and}\ z\in\mathbb{R}\}\ .
\end{equation*}
For all sufficiently large $m$, we foliate a portion of $C_m$ by minimal surfaces as follows. For all $m$, and for all $t\in[-3,3]$, let $c_{m,t}$ denote the curve obtained upon translating $c_m$ by a distance $t$ in the $z$-direction. By Lemma \ref{lemma:ImplicitFunctionTheorem} again, for all $m$, we may suppose that there exists a $C^{2,\alpha}$ family $(S_{m,t})_{t\in[-3,3]}$ of embedded surfaces such that, for all $t$, $S_{m,t}$ is bounded by $c_{m,t}$ and minimal with respect to $g_m$. We claim that the family $(S_{m,t})_{t\in[-3,3]}$ also foliates a subset $U_m$, say, of $C_m\cap4\mathbb{B}^3$. Indeed, for all sufficiently large  $m$ and for all $t$, $S_{m,t}$ is the graph of some function $u_{m,t}$, say. For all such $m$, consider now the function $\tilde{u}_m(x,t):=(x,u_{m,t}(x))$. The sequence $(\tilde{u}_m)_{m\in\mathbb{N}}$ converges smoothly towards the limit $\tilde{u}_\infty(x,t):=(x,t)$. Since $\tilde{u}_\infty$ is a diffeomorphism, so too is $\tilde{u}_m$ for sufficiently large $m$, and it follows that $(S_{m,t})_{t\in[-3,3]}$ foliates some subset of $C_m\cap4\mathbb{B}^3$, as asserted. This yields the desired foliation. Note, in addition, that, for all $m$, we may suppose that $U_m$ contains $C_m\cap2\mathbb{B}^3$.

We now show that, for sufficiently large $m$, $S_m$ is the only minimal surface in $C_m\cap2\mathbb{B}^3$ bounded by $c_m$. Indeed, suppose the contrary, and let $m$ be such that there exists a minimal surface $S_m'\neq S_m$ contained in $C_m\cap2\mathbb{B}^3$ and bounded by $c_m$. By the preceding discussion, we may suppose that the family $(S_{m,t})_{t\in[-3,3]}$ foliates a subset of $C_m$ containing $C_m\cap 2\mathbb{B}^3$. Let $t_-,t_+\in[-3,3]$ denote respectively the minimal and maximal values of $t$ such that $S_{m,t}$ meets $S_m'$. Since $S_m'\neq S_m$, either $t_-<0$ or $t_+>0$. If $t_-<0$, then $S_m'$ meets $S_{m,t_-}$ tangentially at some interior point of $S_m'$. Furthermore, since $S_{m,t_-}$ meets $\partial C_m$ transversally, the point of intersection is also an interior point of $S_{m,t_-}$.
It follows by the strong maximum principle (see Section $1.7$ of \cite{ColdingMinicozzi2011}) that $S_m'=S_{m,t_-}$ and, in particular, $c_m=c_{m,t_-}$, which is absurd. The case where $t_+>0$ is addressed in a similar manner, and it follows that $S_m$ is the only minimal surface in $C_m\cap 2\mathbb{B}^3$ bounded by $c_m$, as desired.

It remains only to show that, for sufficiently large $m$, any minimal surface in $2\mathbb{B}^3$ bounded by $c_m$ is also contained in $C_m$. To this end, fix some large $m$, and note that we may suppose that $(4\mathbb{B}^3,g_m)$ is geodesically convex. We construct an $\epsilon$-convex fattening of the convex hull of $c_m$ as follows. Let $\epsilon>0$ be such that the geodesic curvature of $c_\infty$ is everywhere bounded below by $4\epsilon$. Let $K_m$ denote the intersection of all Euclidean balls in $\mathbb{R}^3$ of radius $1/2\epsilon$ which contain $c_m$. Note that we may suppose that $c_m$ lies along the boundary of $K_m$ and that $K_m$ is contained in $C_m$. Note that, as $\epsilon$ tends to zero, $K_m$ converges towards the convex hull of $c_m$. Hence, for $\epsilon$ sufficiently small, $K_m\subset C_m$.

Since $K_m$ is $(2\epsilon)$-convex with respect to the Euclidian metric, we may suppose that it is $\epsilon$-convex with respect to $g_m$. Let $d_m:4\mathbb{B}^3\rightarrow\mathbb{R}$ denote the distance in $X$ to $K_m$ with respect to $g_m$. By Lemma \ref{lemma:ConvexityOfDistance}, we may suppose that this function is strictly convex. Recall that the restriction of any convex function to any minimal surface is subharmonic (the smooth case is proven in \cite{Jost2017}, and the general case follows from the fact that the supremum of a family of subharmonic functions is subharmonic. We refer the reader to \cite{Greene_Wu_Indiana,Greene_Wu_Ann.ENS} for details). It now follows by the maximum principle that $d_m$ attains its maximum over $S_m'$ along the boundary. However, by construction, $d_m$ vanishes over $c_m$, it thus vanishes over the whole of $S_m'$. It follows that
\begin{equation*}
S_m' \subseteq K_m \subseteq C_m\cap 2\mathbb{B}^3\ ,
\end{equation*}
and this completes the proof.
\end{proof}

A stronger uniqueness result may be obtained by restricting attention to minimal \emph{disks}. For all $p\in\mathbb{S}^2$, and for all $r$, let $c_{p,r}$ denote the geodesic circle of radius $r$ about $p$ in $(\mathbb{S}^2,\overline{g})$.
\begin{lemma}
\label{lemma:DegreeEqualsOne}
There exists $r_0>0$ such that, for all $p$, and for all $r<r_0$, $c_{p,r}$ bounds a unique minimal disk $D_{p,r}$ in $2\mathbb{B}^3$. Furthermore, upon reducing $r_0$ if necessary, we may suppose that $D_{p,r}$ is strictly stable.
\end{lemma}

\begin{proof}To prove existence and strict stability, consider sequences $(r_m)_{m\in\mathbb{N}}$ and $(p_m)_{m\in\mathbb{N}}$ converging respectively $0$ and to some point $p_\infty\in\mathbb{S}^2$, say. For all $m$, consider the translated and rescaled curve
\begin{equation*}
c_m := \frac{1}{r_m}\big(c_{p_m,r_m} - p_m\big)\ ,
\end{equation*}
and consider the translated and rescaled metric
\begin{equation*}
\overline{g}_m(x) := \overline{g}_m\big(r_m\cdot(x+p_m)\big)\ .
\end{equation*}
The sequence $(\overline{g}_m)_{m\in\mathbb{N}}$ converges smoothly towards the Euclidian metric, whilst the sequence $(c_m)_{m\in\mathbb{N}}$ converges towards the unit circle in some plane. It follows by Lemma \ref{lemma:MinimalPerturbationExistence} that, for sufficiently large $m$, $c_m$ bounds a unique disk in $2\mathbb{B}^3$ which is minimal and strictly stable with respect to $\overline{g}_m$. Existence and strict stability now follow by invariance of minimality under rescaling.

To show uniqueness, by Lemma \ref{lemma:MinimalPerturbationExistence}, it suffices to show that, for sufficiently small $r$, any minimal disk in $(\closedthreeball,\overline{g})$ bounded by $c_{p,r}$ is contained in $B_{2r}(p)$.

Consider sequences $(p_m)_{m\in\mathbb{N}}\in\mathbb{S}^2$ and $(r_m)_{m\in\mathbb{N}}>0$ converging respectively to $p_\infty$ and $0$. In order to simplify notation in what follows, we will suppose that, for all $m$, $p_m=p_\infty=:p$. For all $m$, denote $c_m:=c_{p,r_m}$, and let $D_m$ be a minimal disk in $(\closedthreeball,\overline{g})$ bounded by $c_m$. Upon extracting a subsequence, we suppose that $(D_m)_{m\in\mathbb{N}}$ Hausdorff converges towards some limit set $D_\infty$, say, in $\closedthreeball$. Note that $D_\infty$ is connected and contains $p$.

We will show that $D_\infty=\{p\}$. To this end, we apply the theory of varifolds introduced  in \S \ref{ss.gmt_varifouille}. For all finite $m$, let $V_m$ denote the associated varifold of $D_m$ and let $M(V_m)$ denote its mass. By Proposition \ref{prop:area}, for all $m$,
\begin{equation*}
M(V_m) \leq K_m-2\pi\ ,
\end{equation*}
where $K_m$ here denotes the total curvature of $c_m$. Hence
\begin{equation*}
\limsup_{m\rightarrow\infty}M(V_m) \leq 2\pi + \limsup_{m\rightarrow\infty} K_m < \infty\ .
\end{equation*}
The sequence $(V_m)_{m\in\mathbb{N}}$ thus has uniformly bounded mass, and we may therefore assume that it converges weakly to some limit $V_\infty$, say.

We now claim that $V_\infty$ is stationary with respect to $\overline{g}$. Indeed, for any vector field $X$ the first variation of $V_\infty$ with respect to $X$ satisfies (see \S \ref{ss.gmt_varifouille})
\begin{equation*}
\delta V_\infty(X) = \lim_{m\rightarrow\infty}\delta V_m(X)\ .
\end{equation*}
However by minimality, for all $m$, the first variation formula \eqref{eqn:FirstVariationOfDisk} yields
\begin{equation*}
\delta V_m(X) = \int_{c_m}\langle X,n_m\rangle\dd\ell\ ,
\end{equation*}
where here $n_m$ denotes the outward-pointing, unit conormal vector field over $c_m=\partial D_m$. Hence
\begin{equation*}
\left|\delta V_\infty(X)\right| = \lim_{m\rightarrow\infty}\left|\delta V_m(X)\right| \leq \limsup_{m\rightarrow\infty}\opLength(c_m)\cdot\|X\|_{C^0} = 0\ ,
\end{equation*}
so that $V_\infty$ is stationary, as asserted.

Since $(c_m)_{m\in\mathbb{N}}$ converges in the Hausdorff sense to $\{p\}$, it follows by the monotonicity formula that
\begin{equation*}
\opSupp(V_\infty)\setminus\{p\} = D_\infty\setminus\{p\}\ .
\end{equation*}
Thus, if $D_\infty\neq\{p\}$ then $V_\infty$ is a non-trivial stationary varifold which is an interior tangent to $\mathbb{S}^2$ at $p$. Since $\mathbb{S}^2$ is mean convex, and since, in particular, $V_\infty$ is stationary with respect to every \emph{admissible} vector field, this is absurd by White's geometric maximum principle (Theorem \ref{thm:White}).

Finally, let $d:\closedthreeball\rightarrow\Bbb{R}$ denote the distance to $p$ with respect to $\overline{g}$, and let $r_0>0$ be such that $d$ is convex over $B_{r_0}(p)$. By the above discussion, for sufficiently large $m$, $D_m\subseteq B_{r_0}(p)$. As in the proof of Lemma \ref{lemma:MinimalPerturbationExistence}, the restriction of any convex function to any minimal surface is subharmonic, and it follows by the maximum principle that $d$ attains its maximum over $D_m$ along its boundary $\partial D_m=c_m$. However, for sufficiently large $m$, $c_m\subseteq B_{2r_m}(p)$, so that $D_m\subseteq B_{2r_m}(p)$, and this completes the proof.
\end{proof}

\begin{proof}[Proof of Theorem \ref{thm:MainResultA}]It suffices to apply the $\mathbb{Z}$-valued topological mapping degree developed by White in \cite{White1987}. We first recall its construction. Let $\mathcal{C}$ denote the set of $C^{2,\alpha}$ simple closed curves in $\sphere^2$, let $\mathcal{D}$ denote the space of $C^{2,\alpha}$ embedded minimal disks in $\closedthreeball$ with boundary in $\sphere^2$, and let $\partial:\mathcal{D}\rightarrow\mathcal{C}$ denote the map which sends each disk to its boundary.

We first show that $\partial$ is proper. To this end, let $X$ be a compact subset of $\mathcal{C}$. By Proposition \ref{prop:area}, there exists $C>0$ such that, for all $D\in\partial^{-1}(X)$,
\begin{equation*}
\opArea(D) \leq C\ .
\end{equation*}
It follows by Theorem $0$ of \cite{White1987} that $\partial^{-1}(X)$ is compact, and properness of $\partial$ follows.

Since $\partial$ is proper, it follows by \cite{White1991} that it has a well-defined $\mathbb{Z}$-valued differential topological mapping degree which is characterized as follows. Regular points of $\partial$ are embedded minimal disks with invertible Jacobi operators, and regular values are thus curves $c\in\mathcal{C}$ which only bound such embedded minimal disks. By invertibility, for any regular value $c$, $\partial^{-1}(\{c\})$ is discrete and, by compactness, it is finite. Finally, given an embedded minimal disk $D\subseteq(\closedthreeball,\overline{g})$, we define its signature by
\begin{equation*}
\label{eqn:ContributionToDegree}
\sigma(D) := (-1)^{\opMI(D)}\ ,
\end{equation*}
where here $\opMI(D)$ denotes the Morse Index of its Jacobi operator. The degree of $\partial$ then satisfies
\begin{equation*}
\opDeg(\partial) = \sum_{D\in\partial^{-1}(c)}\sigma(D)\ .
\end{equation*}
The content of \cite{White1991} is to show that this degree does not depend on the regular value chosen.

By Lemma \ref{lemma:DegreeEqualsOne}, for all $p$, and for all sufficiently small $r$, the geodesic circle $c_{p,r}$ is a regular value of $\partial$, and $\partial^{-1}(\{c_{p,r}\})$ consists of a single element which, in addition, is strictly stable. It follows that
\begin{equation*}
\opDeg(\partial) = 1\ .
\end{equation*}
Now let $c$ be a smooth, simple, closed curve in $\sphere^2$ such that
\begin{equation*}
K_c \leq 4\pi\ .
\end{equation*}
By Theorem \ref{thm:Nitsche}, every minimal disk $D$ bounded by $c$ is strictly stable and satisfies
\begin{equation*}
\sigma(D) = 1\ .
\end{equation*}
It follows that $c$ is a regular value of $\partial$ and, furthermore,
\begin{equation*}
\#\partial^{-1}(c) = \sum_{D\in\partial^{-1}(\{c\})}1 = \opDeg(\partial) = 1\ .
\end{equation*}
Existence and uniqueness follows, and this completes the proof.
\end{proof}

\bibliographystyle{alpha}
\bibliography{references}

\begin{thebibliography}{EWW02}

\bibitem[ALLS26]{AlvarezLefeuvreLoweSmith}
S.~Alvarez, T.~Lefeuvre, B.~Lowe, and G.~Smith.
\newblock Foliated plateau problems, surface radon transforms, and boundary
  area rigidity in dimension three.
\newblock {\em Preprint}, 2026.

\bibitem[AM10]{alexakis2010renormalized}
S.~Alexakis and R.~Mazzeo.
\newblock Renormalized area and properly embedded minimal surfaces in
  hyperbolic 3-manifolds.
\newblock {\em Comm. Math. Phys.}, 297(3):621--651, 2010.

\bibitem[And83]{Anderson83}
M.~Anderson.
\newblock Complete minimal hypersurfaces in hyperbolic {$n$}-manifolds.
\newblock {\em Comment. Math. Helv.}, 58(2):264--290, 1983.

\bibitem[B\"82]{Bohme82}
R.~B\"ohme.
\newblock New results on the classical problem of {P}lateau. {O}n the existence
  of many solutions.
\newblock In {\em Bourbaki {S}eminar, {V}ol. 1981/1982}, volume 92-93 of {\em
  Ast\'erisque}, pages 1--20. Soc. Math. France, Paris, 1982.

\bibitem[BdC80]{BarbosaDoCarmo1980}
J.~L. Barbosa and M.~do~Carmo.
\newblock Stability of minimal surfaces and eigenvalues of the laplacian.
\newblock {\em Math. Z.}, 173:13--28, 1980.

\bibitem[Bes87]{Besse1987}
A.~Besse.
\newblock {\em Einstein Manifolds}, volume~10 of {\em Ergebnisse der Mathematik
  und ihrer Grenzgebiete}.
\newblock Springer-Verlag, 1987.

\bibitem[CE08]{CheegerEbin2008}
J.~Cheeger and D.~G. Ebin.
\newblock {\em Comparison theorems in {R}iemannian geometry}.
\newblock AMS Chelsea Publishing, Providence, RI, 2008.
\newblock Revised reprint of the 1975 original.

\bibitem[CGH00]{Calderbank_Gauduchon_Herzlich}
D.~Calderbank, P.~Gauduchon, and M.~Herzlich.
\newblock Refined {K}ato inequalities and conformal weights in {R}iemannian
  geometry.
\newblock {\em J. Funct. Anal.}, 173(1):214--255, 2000.

\bibitem[CM11]{ColdingMinicozzi2011}
T.~H. Colding and W.~P. Minicozzi.
\newblock {\em A Course in Minimal Surfaces}.
\newblock Number 121 in Graduate Studies in Mathematics. American Mathematical
  Society, 2011.

\bibitem[EWW02]{EkholmWhiteWienholtz2002}
T.~Ekholm, B.~White, and D.~Wienholtz.
\newblock Embeddedness of minimal surfaces with total boundary curvature at
  most {$4\pi$}.
\newblock {\em Ann. of Math. (2)}, 155(1):209--234, 2002.

\bibitem[Gro91a]{Gromov-91-1}
M.~Gromov.
\newblock Foliated {P}lateau problem. {I}. {M}inimal varieties.
\newblock {\em Geom. Funct. Anal.}, 1(1):14--79, 1991.

\bibitem[Gro91b]{Gromov-91-2}
M.~Gromov.
\newblock Foliated {P}lateau problem. {II}. {H}armonic maps of foliations.
\newblock {\em Geom. Funct. Anal.}, 1(3):253--320, 1991.

\bibitem[GT01]{GilbargTrudinger2001}
D.~Gilbarg and N.~S. Trudinger.
\newblock {\em Elliptic partial differential equations of second order}.
\newblock Classics in Mathematics. Springer-Verlag, Berlin, 2001.
\newblock Reprint of the 1998 edition.

\bibitem[GW79]{Greene_Wu_Ann.ENS}
R.~E. Greene and H.~Wu.
\newblock {$C\sp{\infty }$}\ approximations of convex, subharmonic, and
  plurisubharmonic functions.
\newblock {\em Ann. Sci. \'Ecole Norm. Sup. (4)}, 12(1):47--84, 1979.

\bibitem[GW73]{Greene_Wu_Indiana}
R.~E. Greene and H.~Wu.
\newblock On the subharmonicity and plurisubharmonicity of geodesically convex
  functions.
\newblock {\em Indiana Univ. Math. J.}, 22:641--653, 1972/73.

\bibitem[HLS26]{HuangLoweSeppi26}
Z.~Huang, B.~Lowe, and A.~Seppi.
\newblock Uniqueness and non-uniqueness for the asymptotic {P}lateau problem in
  hyperbolic space.
\newblock {\em Proc. Lond. Math. Soc. (3)}, 132(1):Paper No. e70121, 31, 2026.

\bibitem[Jos17]{Jost2017}
J.~Jost.
\newblock {\em Riemannian Geometry and Geometric Analysis}.
\newblock Universitext. Springer, Cham, 7th edition, 2017.

\bibitem[Kat72]{Kato}
T.~Kato.
\newblock Schr\"odinger operators with singular potentials.
\newblock {\em Israel J. Math.}, 13:135--148, 1972.

\bibitem[LJ94]{Li-Jost}
X.~Li-Jost.
\newblock Uniqueness of minimal surfaces in {E}uclidean and hyperbolic
  {$3$}-space.
\newblock {\em Math. Z.}, 217(2):275--285, 1994.

\bibitem[Low21]{lowe2021deformations}
B.~Lowe.
\newblock Deformations of totally geodesic foliations and minimal surfaces in
  negatively curved 3-manifolds.
\newblock {\em Geom. Funct. Anal.}, 31(4):895--929, 2021.

\bibitem[MK25]{marx2024inverse}
J.~Marx-Kuo.
\newblock An inverse problem for renormalized area: determining the bulk metric
  with minimal surfaces.
\newblock {\em New Zealand J. Math.}, 56:69--124, 2025.

\bibitem[MY82]{MeeksYau1982}
W.~M. Meeks and S.-T. Yau.
\newblock The existence of embedded minimal surfaces and the problem of
  uniqueness.
\newblock {\em Math. Z.}, 179:151--168, 1982.

\bibitem[Nit68]{Nitsche68}
J.~C.~C. Nitsche.
\newblock Contours bounding at least three solutions of {P}lateau's problem.
\newblock {\em Arch. Rational Mech. Anal.}, 30:1--11, 1968.

\bibitem[Nit73]{Nitsche1973}
J.~C.~C. Nitsche.
\newblock A new uniqueness theorem for minimal surfaces.
\newblock {\em Arch. Rational Mech. Anal.}, 52:319--329, 1973.

\bibitem[Nit89]{Nitsche_book}
J.~C.~C. Nitsche.
\newblock {\em Lectures on minimal surfaces. {V}ol. 1}.
\newblock Cambridge University Press, Cambridge, 1989.
\newblock Introduction, fundamentals, geometry and basic boundary value
  problems, Translated from the German by Jerry M. Feinberg, With a German
  foreword.

\bibitem[Ros93]{Rosenberg1993}
H.~Rosenberg.
\newblock Hypersurfaces of constant curvature in space forms.
\newblock {\em Bull. Sci. Math.}, 117(2):211--239, 1993.

\bibitem[Sep16]{Seppi16}
A.~Seppi.
\newblock Minimal discs in hyperbolic space bounded by a quasicircle at
  infinity.
\newblock {\em Comment. Math. Helv.}, 91(4):807--839, 2016.

\bibitem[Sim68]{Simons68}
J.~Simons.
\newblock Minimal varieties in riemannian manifolds.
\newblock {\em Ann. of Math. (2)}, 88:62--105, 1968.

\bibitem[Sim83]{simon1983lectures}
L.~Simon.
\newblock {\em Lectures on geometric measure theory}, volume~3 of {\em
  Proceedings of the Centre for Mathematical Analysis, Australian National
  University}.
\newblock Australian National University, Centre for Mathematical Analysis,
  Canberra, 1983.

\bibitem[Som04]{Som04}
T.~Soma.
\newblock Existence of least area planes in hyperbolic 3-space with co-compact
  metric.
\newblock {\em Topology}, 43(3):705--716, 2004.

\bibitem[Som05]{soma05}
T.~Soma.
\newblock Least area planes in gromov hyperbolic 3-spaces with co-compact
  metric.
\newblock {\em Geometriae Dedicata}, 112(1):123--128, 2005.

\bibitem[Spi79]{spivak1979comprehensive3}
M.~Spivak.
\newblock {\em A Comprehensive Introduction to Differential Geometry, Volume
  III}, volume~3.
\newblock Publish or Perish, Inc., Wilmington, Delaware, 2nd edition, 1979.

\bibitem[Uhl83]{Uhlenbeck83}
K.~Uhlenbeck.
\newblock Closed minimal surfaces in hyperbolic {$3$}-manifolds.
\newblock In {\em Seminar on minimal submanifolds}, volume 103 of {\em Ann. of
  Math. Stud.}, pages 147--168. Princeton Univ. Press, Princeton, NJ, 1983.

\bibitem[Whi87]{White1987}
B.~White.
\newblock Curvature estimates and compactness theorems in $3$-manifolds for
  surfaces that are stationary for parametric elliptic functionals.
\newblock {\em Invent. Math.}, 88:243--256, 1987.

\bibitem[Whi91]{White1991}
B.~White.
\newblock The space of minimal submanifolds for varying riemannian metrics.
\newblock {\em Indiana Univ. Math. J.}, 40(1):161--200, 1991.

\bibitem[Whi10]{white2010maximum}
B.~White.
\newblock The maximum principle for minimal varieties of arbitrary codimension.
\newblock {\em Comm. Anal. Geom.}, 18(3):421--432, 2010.

\end{thebibliography}
\end{document}